\documentclass[a4paper,fleqn]{cas-sc}

\usepackage[authoryear,longnamesfirst]{natbib}

\usepackage{amsthm}          % \theoremstyle, \newtheorem
\usepackage{mathrsfs}        % \mathscr
\usepackage{caption}         % \captionof for non-floating figures

\theoremstyle{plain}
\newtheorem{theorem}{Theorem}[section]
\newtheorem{lemma}[theorem]{Lemma}

\theoremstyle{definition}
\newtheorem{definition}[theorem]{Definition}
\newtheorem{example}[theorem]{Example}

\theoremstyle{remark}
\newtheorem{remark}[theorem]{Remark}

\newcommand{\dd}{\,\mathrm{d}}
\newcommand{\Lop}{\mathcal{F}}

\newcommand{\pdom}{\overline\Omega}

\newcommand{\intr}{\mathring}

\begin{document}
\let\WriteBookmarks\relax
\def\floatpagepagefraction{1}
\def\textpagefraction{.001}

% Short title
\shorttitle{Natural superconvergence points for spline finite elements}

% Short author
\shortauthors{P. Yang and Z. Zhang}

% Main title of the paper
\title [mode = title]{Natural superconvergence points and asymptotic expansions
       for spline finite elements in one dimension}

% Title footnote (funding acknowledgment)
\tnotemark[1]
\tnotetext[1]{This work was partially supported by the National Natural Science Foundation of China (Grant No.~12501537).}

% First author
%
% Options: Use if required
% eg: \author[1,3]{Author Name}[type=editor,
%       style=chinese,
%       auid=000,
%       bioid=1,
%       prefix=Sir,
%       orcid=0000-0000-0000-0000,
%       facebook=<facebook id>,
%       twitter=<twitter id>,
%       linkedin=<linkedin id>,
%       gplus=<gplus id>]

\author[1]{Peng Yang}
\cormark[1]
%\fnmark[1]
\ead{pyang@uestc.edu.cn}

\affiliation[1]{organization={University of Electronic Science and Technology of China},
	addressline={School of Mathematical Sciences}, 
	city={Chengdu},
%          citysep={}, % Uncomment if no comma needed between city and postcode
	postcode={611731}, 
	state={Sichuan},
	country={China}}

% Second author
\author[2]{Zhimin Zhang}
%\cormark[2]
\ead{ag7761@wayne.edu}

\affiliation[2]{organization={Wayne State University},
	addressline={Department of Mathematics}, 
	city={Detroit},
%          citysep={}, % Uncomment if no comma needed between city and postcode
	postcode={48202}, 
	state={Michigan},
	country={USA}}

% Corresponding author text
\cortext[1]{Corresponding author}

%% Footnote text
%\fntext[1]{}

% Here goes the abstract
\begin{abstract}
We study the natural superconvergence points and asymptotic expansions of one-dimensional
spline finite element approximations.  For a spline space of degree $k$ and  any smoothness $0\le\mu\le k-1$, we prove that the
$s$-th derivative of the error exhibits enhanced convergence of
order $O(h^{k+2-s})$ at points where $k-s$ is even, provided the
mesh is symmetric within a region of size $Ch|\ln h|$ around the
point.  This condition is known to be optimal for the cases of low derivative
order $s=0,1$; the present
analysis shows that the same local condition is sufficient for all admissible
$s$. Moreover, by  expanding the error  in Legendre
polynomials,  a closure theorem determines the leading-order Legendre
coefficients (the asymptotic expansion of the error) by combining the Galerkin orthogonality
with the superconvergence conditions.  For
$\mu=k-1$ (B-splines)  and $\mu=k-2$, the Galerkin orthogonality conditions vanish and the
coefficients are determined solely by the superconvergence conditions.  The  asymptotic expansion can be expressed through a simple antiderivative
recurrence on Legendre polynomials. The resulting polynomial's zeros  encode the complete set of
superconvergence points for all derivative orders.
 Numerical experiments for selected $(k,\mu)$ pairs
confirm the theoretical predictions.
\end{abstract}

% Use if graphical abstract is present
%\begin{graphicalabstract}
%\includegraphics{}
%\end{graphicalabstract}

% Research highlights
\begin{highlights}
\item A unified theory of the natural superconvergence points of one-dimensional spline finite element spaces $S_h^{k,\mu}$ is developed for all derivative orders and smoothness levels.
\item A closure theorem determines the leading-order asymptotic expansion of the error by combining Galerkin orthogonality with nodal superconvergence conditions.
\item For maximal smoothness (B-splines), closed-form error polynomials obtained by an antiderivative recurrence encode all superconvergence points, with explicit asymptotic expansions.
\end{highlights}

% Keywords
% Each keyword is seperated by \sep

\maketitle

% ===========================================================================
% 1. Introduction
% ===========================================================================
\section{Introduction}
\label{sec:intro}

Superconvergence is a well-known and intriguing phenomenon in finite
element analysis: while the global error is limited by the
polynomial degree and the mesh size, there exist specific points
within each element where the numerical solution or its derivatives
converge at a significantly higher rate than the global norm would
predict.  For standard $C^0$ Lagrange elements, the classical
theory, developed in the 1970s, established that the gradient is
superconvergent at the Gauss--Legendre points
\cite{DouglasDupont1973,LesaintZlamal1979} and the function
value at the Gauss--Lobatto points
\cite{Zlamal1978,DouglasDupontWheeler1974}.  These results have since
been extended to a wide variety of element types and problem
classes~\cite{Thomee1977,BrandtsKrizek2001,HeWM2017,Zhu1989,Chen1995,GuoZhangZou2018}.

In the past two decades, isogeometric analysis (IGA) has emerged
as a powerful paradigm that bridges computer-aided design and
finite element analysis~\cite{Hughes2005}.  The solution space in
IGA consists of spline spaces, that is, the spaces $S_h^{k,\mu}$
of degree $k$ and global smoothness $\mu$ studied in this paper.
Moreover, IGA commonly employs tensor-product grids in higher
dimensions, where the one-dimensional superconvergence structure is
carried over to each coordinate direction
\cite{Fahrendorf2018,WangQiLi2021}.  Consequently, a complete
understanding of the superconvergence points for one-dimensional
spline spaces serves as a foundation for multidimensional analysis.

A landmark reference on this subject is the monograph of
Wahlbin~\cite{Wahlbin1996}.  Inside a table of his monograph (summarized
in our Table~\ref{tab:wahlbin_summary}),
Wahlbin systematically summarized the superconvergence points for the
spline spaces $S_h^{k,\mu}$.
For each combination of degree $k$ and lower  smoothness $\mu$, he gave the
asymptotic expansion of the error on a reference element in terms
of Legendre polynomials, and identified the superconvergence
points as the zeros of the leading-order terms of these expansions.
An important aspect of Wahlbin's analysis is that the
superconvergence requires only a very mild local condition:
the mesh needs to be uniform only within a region of size
$Ch|\ln h|$ around the point in question.  This indicates that
superconvergence is a genuinely local phenomenon, as further
analyzed in~\cite{SchatzSloanWahlbin1996}.  What remains
absent from Wahlbin's table, however, is the superconvergence
behavior for higher-order derivatives (derivative order $s\ge2$) and for spline
spaces of high smoothness ($\mu\ge1$). 
\begin{table}[htbp]
	\centering
	\caption{Superconvergence points summarized in~\cite[Chapter~1]{Wahlbin1996}, where $L_k(t)$ is the Legendre polynomial of degree $k$.}
	\label{tab:wahlbin_summary}
	\small
	\begin{tabular}{p{0.28\textwidth}|p{0.33\textwidth}|p{0.33\textwidth}}
		$\mu: k$ & Function values (derivative order $s=0$) & First derivative (derivative order $s=1$) \\ \hline
		$\mu=0$: Any $k$, completely general meshes &
		$O(h^{2k})$ at meshpoints;
		$O(h^{k+2})$ at the interior zeros of $L'_k(t)$ &
		$O(h^{k+1})$ at zeros of $L_k(t)$ \\ \hline
		$\mu=1$: $k$ odd, meshes uniform in a $Ch|\ln h|$ neighborhood &
		$O(h^{k+2})$ at the $k-1$ zeros of $Q(t):= L_{k-1}(t)-\frac{L'_{k-1}(1)}{L'_{k+1}(1)} L_{k+1}(t)$ &
		$O(h^{k+1})$ at meshpoints and midpoints, and at additional $k-3$ zeros of $Q'(t)$ \\ \hline
		$\mu=1$: $k$ even, meshes as above &
		$O(h^{k+2})$ at meshpoints, and at the interior zeros of $L'_k(t)$ &
		$O(h^{k+1})$ at zeros of $L_k(t)$ \\ \hline
		$\mu=2$: $k=3$ (smoothest cubics), meshes as above &
		$O(h^{k+2})$ at two points, same as for Hermite cubics ($\mu=1,\;k=3$) above &
		$O(h^{k+1})$ at meshpoints and midpoints \\ \hline
		$\mu\ge1$: $k$ odd, general meshes symmetric in a $Ch|\ln h|$ neighborhood &
		Not known &
		$O(h^{k+1})$ at meshpoints and midpoints (incomplete) \\ \hline
		$\mu\ge1$: $k$ even, general meshes symmetric as above &
		$O(h^{k+2})$ at meshpoints and midpoints (incomplete) &
		Not known \\ \hline
	\end{tabular}
\end{table}  

Various subsequent works have made progress in related directions.
Babu\v{s}ka et al.~\cite{BabuskaStrouboulis1996} used computer-assisted
proof techniques to establish the existence of superconvergence
points and to reveal the distribution patterns of derivative
superconvergence.  Anitescu et al.~\cite{Anitescu2015}
conducted computational studies on B-spline based collocation
methods and reported superconvergence point distributions; this line of research was later extended in~\cite{Montardini2017, XuKangChenZhang2025}.  For a partial survey of the literature on natural
superconvergence points, see~\cite{BrandtsKrizek2001,Babuska1007,Lin2004,Lin2008,cao2023superconvergence,cao2022c}.
Nevertheless, a unified theoretical framework and closed-form
characterization of the superconvergence points for spline spaces with smoothness $\mu\ge1$ and derivative order
$s\ge2$ of the error has remained unavailable prior to the present work.

In this paper, we provide such a framework for the spline spaces
$S_h^{k,\mu}$ on quasi-uniform meshes.  Our first important result (Theorem~\ref{thm:local_super})
is a symmetry-based superconvergence estimate.  When $k-s$ is even,
the $s$-th derivative of the error is $O(h^{k+2-s})$ at points where
the mesh is locally uniform within a region of size $Ch|\ln h|$.
This covers both element midpoints (for $s\le k$) and meshpoints (for
$s\le\mu$).  We expand the error on each element in Legendre polynomials and apply the superconvergence estimate at the midpoints. This shows that  roughly half of the expansion coefficients are one order higher
than expected and can be discarded.  The remaining coefficients are then determined by
combining the Galerkin orthogonality with the nodal superconvergence
conditions, yielding the closure theorem (Theorem~\ref{thm:closure}) that
applies to all admissible $(k,\mu)$.

When $\mu=k-1$ (the B-spline case) or $\mu=k-2$, the Galerkin
orthogonality constraints are absent and the expansion coefficients
are determined by the superconvergence conditions alone.  As a
result, the leading-order polynomial of the $s$-th derivative has a particularly simple
closed form. It is obtained by applying an antiderivative
operator $k-s$ times  to the Legendre polynomial of degree one. The resulting
polynomial is of degree $k-s+1$, and  its zeros   have a simple structure: for even $k-s$ they are the
element endpoints and midpoint; for odd $k-s$ they are two symmetric
interior points $\pm a_{k-s}$.   Moreover, $a_{k-s}$ converges to
$1/2$ as $k-s\to\infty$, and we derive its precise asymptotic
expansion.  Using the closure theorem together with these explicit
results, we compute the leading-order polynomials (asymptotic expansions) for all
admissible $(k,\mu)$ pairs, thereby filling the entries left open in
Wahlbin's table (see also Table~\ref{tab:wahlbin_summary}) and completing the description of superconvergence points for all derivative orders.

The paper is organized as follows.  Section~\ref{sec:prelim} describes the model
problem and introduces the necessary tools.
Section~\ref{sec:super} presents the symmetry-induced superconvergence points.  Section~\ref{sec:asymptotic} develops
the asymptotic expansion on a reference element, including the closure
theorem and the explicit results for the maximal smoothness cases.
Section~\ref{sec:numerical} reports numerical experiments, and Section~\ref{sec:conclusion} concludes
the paper.  The Appendix supplies a key estimate used in the proof of the sharp local approximation error estimate.

Throughout this paper, $C$ denotes a generic positive constant independent of the mesh size $h$ (which may be arbitrarily small), and the value of $C$ may vary in different contexts.  In particular,
$A=O(h^m)$ means $|A|\le C h^m$.
% ===========================================================================
% 2. Preliminaries and problem setup
% ===========================================================================
\section{Preliminaries}
\label{sec:prelim}

% ---------------------------------------------------------------------------
% 2.1 Model problem and finite element discretisation
% ---------------------------------------------------------------------------
\subsection{Model problem and spline finite element discretisation}
\label{sec:model}

Consider the two-point boundary value problem with homogeneous Dirichlet
conditions: find $u$ such that
\begin{equation}\label{eq:model}
\left\{
\begin{aligned}
&-u''(x) = f(x), \quad x\in\Omega := (a,b),\\
&u(a)=u(b)=0,
\end{aligned}
\right.
\end{equation}
where $f$ is sufficiently smooth.  The variational formulation reads: find
$u\in H_0^1(\Omega)$ such that
$a(u,v) = (f,v)$ for all $v\in H_0^1(\Omega)$, where
$a(u,v)=\int_\Omega u'v'\,\dd x$ and $(\cdot,\cdot)$ denotes the
$L^2(\Omega)$ inner product.

Let $\{x_i\}_{i=1}^{N+1}$ be a quasi-uniform partition of $\pdom=[a,b]$ with
mesh size $h:=\max_i (x_{i+1}-x_i)$, and set $I_i=[x_i,x_{i+1}]$.
Define the polynomial spline space
\begin{equation*}
S_h^{k,\mu}=S_h^{k,\mu}(\Omega)
= \bigl\{\chi\in C^\mu(\Omega):\;
\chi|_{I_i}\in\mathcal{P}_k(I_i) \,\,\, \forall i \bigr\},
\end{equation*}
where $\mathcal{P}_k(I_i)$ denotes the space of polynomials of degree at
most $k$ on $I_i$.  The parameter $\mu\ge 0$ characterizes the global
smoothness: $\chi\in S_h^{k,\mu}$ possesses continuous derivatives up to
order $\mu$ throughout $\Omega$.  We assume $0\le\mu\le k-1$ for
non-trivial spline spaces.  Well-known examples include:
\begin{itemize}
  \item $\mu=0$: $C^0$ Lagrange elements;
  \item $\mu=1$, $k=3$: Hermite cubics;
  \item $\mu=k-1$: B-splines (maximal smoothness).
\end{itemize}

The Galerkin finite element solution $u_h\in S_h^{k,\mu}$ satisfies
\begin{equation}\label{eq:FEM}
a(u-u_h,\chi)=0\qquad\forall\,\chi\in V_h:=S_h^{k,\mu}\cap H_0^1(\Omega),
\end{equation}
with $u_h(a)=u_h(b)=0$, i.e., $u_h\in V_h$.
\begin{remark}
	We work with the simple Poisson model~\eqref{eq:model} for the analysis.
	By Wahlbin's  theorems~\cite[Chapter 1]{Wahlbin1996}, the
	difference $u_h - \tilde u_h$ between the finite element
	solution $ \tilde u_h$ of a general second-order problem (with variable
	coefficients and lower-order terms) and the solution $u_h$ of the
	simple model satisfies $\|(u_h - \tilde u_h)'\|_{L^\infty}=O(h^{k+1})$ and
	$\|u_h - \tilde u_h\|_{L^\infty}=O(h^{k+2})$.  For derivatives of order $s\ge2$,
	the inverse inequality gives $|(u_h - \tilde u_h)^{(s)}(x_0)|=O(h^{k+2-s})$
	for \emph{any} point $x_0$, provided $s\le\mu$ when $x_0$ is a mesh node,
	or $s\le k$ when $x_0$ lies in the interior of an element.
	Consequently, the entire superconvergence analysis extends  to
	general second-order problems, and we henceforth restrict attention
	to~\eqref{eq:model} without loss of generality.
\end{remark}

% ---------------------------------------------------------------------------
% 2.2 Discrete delta functions and exponential decay
% ---------------------------------------------------------------------------
\subsection{Discrete delta functions and exponential decay}
\label{sec:discrete_delta}

Let $x_0\in\Omega$ be a fixed point.  Consider $\mu\geq 1$.  For an integer $s$ with
$1\le s\le\mu$ when $x_0$ is a mesh point, and $1\le s\le k$ when $x_0$
lies inside an element, define the discrete delta function
$\delta_h=\delta_h^{s,x_0}\in S_h^{k-1,\mu-1}$ by
\begin{equation}\label{eq:discrete_delta}
(\chi,\delta_h)_{L^2(\Omega)} = \chi^{(s-1)}(x_0),\qquad
\forall\,\chi\in S_h^{k-1,\mu-1}.
\end{equation}
\begin{lemma}[Exponential decay]\label{lem:decay}
There exist constants $C,c>0$, independent of $h$ and
$x_0$, such that
\begin{equation}
|\delta_h(x)|\le C\,h^{-s}\,e^{-c|x-x_0|/h},\qquad
\forall x\in\Omega.
\label{eq:delta_decay}
\end{equation}
\end{lemma}

\begin{proof}
We analyze the linear system that determines the coefficients of $\delta_h$.
Let $\{\psi_j\}$ be the B-spline basis of $S_h^{k-1,\mu-1}$; the support of each $\psi_j$
 has width $O(h)$.  Since
$\|\psi_j\|_{L^2(\Omega)}=O(h^{1/2})$, introduce the normalized basis
\[
\tilde\psi_j:=h^{-1/2}\psi_j,\qquad\|\tilde\psi_j\|_{L^2(\Omega)}=1.
\]
The mass matrix $M_{ij}:=(\tilde\psi_i,\tilde\psi_j)$ is symmetric, positive
definite, banded with bandwidth independent of $h$, and, by quasi-uniformity
of the mesh, satisfies the condition number  $\kappa(M)=O(1)$.
Then, the Demko--Moss--Smith theorem~\cite{demko1984decay} supplies a
constant $0<\rho<1$, depending only on the bandwidth and
condition number, such that
\begin{equation}\label{eq,invM}
	|(M^{-1})_{ij}|\le C\,\rho^{|i-j|}\qquad\forall\,i,j.
\end{equation}

Write $\delta_h(x)=\sum_i d_i\tilde\psi_i(x)$.  Applying~\eqref{eq:discrete_delta}
with $\chi=\tilde\psi_j$ gives $M\mathbf{d}=\mathbf{f}$, where
$f_j:=\tilde\psi_j^{(s-1)}(x_0)$.
Only the $O(1)$ indices $j$ whose supports contain $x_0$ contribute a
non-zero $f_j$; denote this index set by
$J_0:=\{j:\operatorname{supp}\tilde\psi_j\ni x_0\}$.  For $j\in J_0$,
the inverse inequality for polynomials yields
\[
|f_j|
\le C\,\|\tilde\psi_j^{(s-1)}\|_{L^\infty(\Omega)}
\le C\,h^{-(s-1)}\|\tilde\psi_j\|_{L^\infty(\Omega)}
\le C\,h^{-(s-1)}h^{-1/2}=C\,h^{-s+1/2}.
\]
From $\mathbf{d}=M^{-1}\mathbf{f}$ and \eqref{eq,invM},
\[
|d_i|\le\sum_{j\in J_0}|(M^{-1})_{ij}|\,|f_j|
\le C\sum_{j\in J_0}\rho^{|i-j|}\,\,h^{-s+1/2}
\le C\,h^{-s+1/2}\,\rho^{|i-i_0|},
\]
where $i_0$ is an index in $J_0$ that minimizes $|i-i_0|$.  Quasi-uniformity of the mesh gives
$|i-i_0|\ge C_i|x_i-x_0|/h$ for some constant $C_i>0$, and therefore
\begin{equation}\label{eq,di}
	|d_i|\le C\,h^{-s+1/2}\,e^{-C_i|\ln\rho||x_i-x_0|/h}.
\end{equation}

For any $x\in\Omega$, let $I(x):=\{i:\operatorname{supp}\tilde\psi_i\ni x\}$.
Because each $\operatorname{supp}\tilde\psi_i$ has width $O(h)$, we have
$|I(x)|=O(1)$.  
Using \eqref{eq,di} and noting that for $i\in I(x)$ we have
$|x_i-x_0|\ge|x-x_0|-Ch$ (since $|x_i-x|\le Ch$), we obtain
\[
|\delta_h(x)|
\le\sum_{i\in I(x)}|d_i|\,\|\tilde\psi_i\|_{L^\infty}
\le\sum_{i\in I(x)} C\,h^{-s+1/2}\,e^{-c_i|\ln\rho||x_i-x_0|/h}h^{-1/2}
\le C\,h^{-s}\,e^{-c|x-x_0|/h}.
\]
  This completes the proof.
\end{proof}

\begin{remark}\label{rem:demko_condition}
The Demko--Moss--Smith theorem guarantees $h$-independent exponential
decay of the inverse entries for a symmetric positive definite banded
matrix whose condition number is bounded independent of $h$.
The mass matrix $M$ satisfies this requirement because the condition number $\kappa(M)=O(1)$.

The global stiffness matrix $A$ with entries
$A_{ij}:=(\psi_i',\psi_j')$, on the other hand, has condition
number $\kappa(A)=O(h^{-2})$.  For such matrices the Demko--Moss--Smith
theorem still applies, but the decay factor satisfies $\rho=1-O(h)$, so
that $h^{s}|\delta_h(x)| \leq C\,\rho^{c|\ln h|}\to O(1)$ as $h\to0$ for $|x-x_0|= O(h|\ln h|)$, yielding no
effective exponential decay.  Consequently, a direct stiffness-matrix
approach cannot match the exponential decay that the mass-matrix-based
discrete delta construction provides. 
\end{remark}
% ---------------------------------------------------------------------------
% 2.3 Localised error decomposition
% ---------------------------------------------------------------------------
\subsection{Localised error decomposition}
\label{sec:decomposition}

Let $B_d(x_0)\subset\Omega$ be a subdomain centered at $x_0$, with diameter
determined by
\begin{equation}\label{eq:d_choice}
d=C^* h|\ln h|,
\end{equation}
where $C^*>0$ is a sufficiently large constant independent of $h$ and $x_0$.
The domain $B_d(x_0)$ is taken to be a union of whole elements.  Let $N_b$
denote the set of mesh nodes on $\partial B_d$.  Define the subspace
\[
\intr S_h^{k,\mu}(B_d)=\bigl\{\chi\in S_h^{k,\mu}(B_d):\;
\chi^{(j)}(y)=0\;\forall y\in N_b,\; j=0,1,\dots,\mu\bigr\},
\]
consisting of spline functions whose values and all derivatives up to order
$\mu$ vanish at the boundary nodes.  Generally, we have
\[
\dim S_h^{k,\mu}(B_d)-\dim\intr S_h^{k,\mu}(B_d)=2(\mu+1).
\]
Define the local spline projection $\Pi_{B_d}u\in S_h^{k,\mu}(B_d)$ by
\begin{equation}\label{eq:local_projection}
\begin{cases}
a_{B_d}(u-\Pi_{B_d}u,\chi)=0, & \forall\,\chi\in\intr S_h^{k,\mu}(B_d),\\[2mm]
(\Pi_{B_d}u)^{(j)}(y)=u^{(j)}(y), & \forall y\in N_b,\; j=0,1,\dots,\mu.
\end{cases}
\end{equation}
The system~\eqref{eq:local_projection} is uniquely solvable: the bilinear
form $a_{B_d}(\cdot,\cdot)$ is coercive on the subspace with vanishing
boundary data, and the boundary conditions determine the remaining degrees
of freedom.

For the error $e:=u-u_h$, we decompose the pointwise derivative at $x_0$ as
\begin{equation}
e^{(s)}(x_0)=(\Pi_{B_d}u-u_h)^{(s)}(x_0)+(u-\Pi_{B_d}u)^{(s)}(x_0).
\label{eq:decomposition}
\end{equation}
Here, $s\le\mu$ when $x_0$ is a mesh node, and $s\le k$ when $x_0$ lies in the interior of an element.
The first term is the projection difference and the second
term is the local approximation error; both are analyzed in
Section~\ref{sec:super}.

% ===========================================================================
% 3. Symmetry-induced superconvergence points
% ===========================================================================
\section{Symmetry-induced superconvergence points}
\label{sec:super}

% ---------------------------------------------------------------------------
% 3.1 Estimation of the projection difference for s >= 2
% ---------------------------------------------------------------------------
\subsection{Estimation of the projection difference for $s\ge 2$}
\label{sec:proj_diff}

In this subsection, we estimate the first term $(\Pi_{B_d}u-u_h)^{(s)}(x_0)$ in~\eqref{eq:decomposition} for $s\ge 2$ using the discrete delta functions $\delta_h$ defined by \eqref{eq:discrete_delta}. 

First, we show a truncation error of the discrete delta functions $\delta_h$.
Let $\{\psi_i\}$ be the B-spline basis of $S_h^{k-1,\mu-1}$ and introduce the normalized basis
$\tilde\psi_j:=h^{-1/2}\psi_j$. Write $\delta_h(x)=\sum_i d_i\tilde\psi_i(x)$.  Let $J_d=\{i:\operatorname{supp}(\tilde\psi_i)\subset B_d\}$ and define the truncation
\begin{equation}
	T_h\delta_h:=\sum_{i\in J_d} d_i\tilde\psi_i,
	\label{eq:Tdelta}
\end{equation}
which evidently belongs to $\mathring{S}_h^{k-1,\mu-1}(B_d)$.

\begin{lemma}[Truncation error]\label{lem:trunc}
For any $M>0$, 
\begin{equation*}\label{eq:trunc_err,L1}
	\|\delta_h-T_h\delta_h\|_{L^1(\Omega)}+
	\|\delta_h-T_h\delta_h\|_{L^2(\Omega)}=O(h^M),
\end{equation*}
provided $d= C^* h|\ln h|$ as in~\eqref{eq:d_choice} with $C^*$ sufficiently large.
\end{lemma}
\begin{proof}
We estimate
\[
\|\delta_h-T_h\delta_h\|_{L^1(\Omega)} 
= \int_{\Omega}\Bigl|\sum_{i\notin J_d}d_i\tilde\psi_i\Bigr|
\le \|\delta_h\|_{L^1(\Omega\setminus B_{d})} 
+ \int_{B_{d}} \sum_{i\notin J_{d}}|d_i|\,|\tilde\psi_i|.
\]
The number of indices $i\notin J_d$ with $\operatorname{supp}(\tilde\psi_i)\cap B_{d}\neq\varnothing$ is $O(1)$.  For any such $i$, every point $x\in\operatorname{supp}(\tilde\psi_i)\cap B_d$ satisfies $|x-x_0|\ge d-Ch = Cd$, i.e., at a distance $Ch|\ln h|$ from $x_0$.  By \eqref{eq,di} and $\|\tilde\psi_i\|_{L^\infty}\le C h^{-1/2}$, we obtain
\[
\int_{B_{d}} \sum_{i\notin J_{d}}|d_i|\,|\tilde\psi_i| \le C h^{-s} e^{-c C^*|\ln h|} = O(h^M),
\]
provided $C^*= (M+s)/c$.
On the other hand, it follows from \eqref{eq:delta_decay} that
\begin{align*}
	\|\delta_h\|_{L^1(\Omega\setminus B_{d})} &\le \int_{|x-x_0|\ge C^* h|\ln h|} C h^{-s} e^{-c|x-x_0|/h}\,dx \\
	&\le C h^{-s}\int_{y\ge C^* |\ln h|}  h e^{-cy}\,dy \\
	&= C\,h^{-s+1} \frac{1}{c }h^{cC^*}\\
	&=O(h^{M+1}).
\end{align*}
The $L^2$ estimate follows analogously. This completes the proof.
\end{proof}

Let the mesh nodes on $\partial B_d$ be taken as endpoints of macro elements~\cite{Yang2025}.  By the element-based B-spline construction of~\cite{Yang2025}, there exists an interpolant $I_h^{\text{spl}}u\in S_h^{k,\mu}(\Omega)$ that matches $u$ and its derivatives up to order $\mu$ at these nodes and satisfies
\begin{equation*}\label{eq,spline interpolant}
	\|u-I_h^{\text{spl}}u\|_{H^1(\Omega)}\leq Ch^k.
\end{equation*}
Extend $\Pi_{B_d}u$ to $\Omega$ by
\[
\Pi^*_{\Omega}u(x):=
\begin{cases}
\Pi_{B_d}u(x), & x\in B_d,\\[2mm]
I_h^{\text{spl}}u(x), & x\in\Omega\setminus B_d.
\end{cases}
\]
On $\partial B_d$, both $\Pi_{B_d}u$ and $I_h^{\text{spl}}u$ coincide with $u$ in value and in derivatives up to order $\mu$ (by~\eqref{eq:local_projection} and the matching property of $I_h^{\text{spl}}u$), consequently $\Pi^*_{\Omega}u \in S_h^{k,\mu}(\Omega)$.  We now estimate the projection difference.

\begin{lemma}[Estimation of the projection difference]\label{lem:proj_diff_sge2}
Let $s\geq 2$.  Then for any $M>0$,
\[
(\Pi_{B_d}u-u_h)^{(s)}(x_0)=O(h^M),
\]
provided $d= C^* h|\ln h|$ as in~\eqref{eq:d_choice} with $C^*$ sufficiently large.  Here, $s\le\mu$ when $x_0$ is a mesh node, and $s\le k$ when $x_0$ lies in the interior of an element.
\end{lemma}

\begin{proof}
Since $x_0\in B_d$, we have $(\Pi_{B_d}u-u_h)^{(s)}(x_0)=(\Pi^*_{\Omega}u-u_h)^{(s)}(x_0)$ by construction.  Set $w:=(\Pi^*_{\Omega}u-u_h)'\in S_h^{k-1,\mu-1}(\Omega)$.  Then, by \eqref{eq:delta_decay}, we have
\[
(\Pi_{B_d}u-u_h)^{(s)}(x_0)=w^{(s-1)}(x_0)=(\delta_h,w).
\]
In addition, it follows from the global and  local spline projections \eqref{eq:FEM}, \eqref{eq:local_projection} that
\begin{equation}\label{eq: orthogonality,w}
	(\chi',w) = 0 \quad \forall\,\chi\in\intr S_h^{k,\mu}(B_d).
\end{equation}
We would like to combine this orthogonality with the representation $(\delta_h,w)$ by choosing $\chi$ such that $\chi'$ approximates $\delta_h$ on $B_d$.  The truncated delta function $T_h\delta_h$ defined in \eqref{eq:Tdelta} belongs to $\mathring{S}_h^{k-1,\mu-1}(B_d)$, but it is not necessarily of the form $\chi'$ with $\chi\in\mathring{S}_h^{k,\mu}(B_d)$.  Consequently, we cannot directly exploit the orthogonality $(\chi',w)=0$ with $T_h\delta_h$ as a test function.  Instead, we construct a corrected truncated delta function $\widetilde\delta_h$ that differs from $T_h\delta_h$ by an $O(h^M)$ perturbation and belongs to the admissible space $W_h:=\{\chi': \chi\in \mathring{S}_h^{k,\mu}(B_d)\}$.
To this end, define 
\[
m:=\int_{B_d}T_h\delta_h\,dx = \int_{\Omega}T_h\delta_h\,dx.
\]
 By \eqref{eq:discrete_delta} with $\chi=1$, the discrete delta function satisfies $\int_{\Omega}\delta_h = 1^{(s-1)}(x_0)=0$ for $s\ge 2$.  Hence, by Lemma~\ref{lem:trunc},
\begin{equation}\label{eq: arbitray low order of m}
	|m| = \Bigl|\int_{\Omega}(T_h\delta_h-\delta_h)\Bigr|
	\le \|\delta_h-T_h\delta_h\|_{L^1(\Omega)} = O( h^{M}).
\end{equation}
Let $\psi_{i_0}$ be a B-spline of $S_h^{k-1,\mu-1}$ such that $x_0\in\operatorname{supp}(\psi_{i_0})$ and define
\begin{equation*}
	\psi_0(x):=\Bigl(\int \psi_{i_0}\Bigr)^{-1} \psi_{i_0}(x),
\end{equation*}
normalized so that $\int\psi_0=1$.  Since $\operatorname{supp}(\psi_{i_0})\subset B_d$ for sufficiently large $C^*$ (because $\operatorname{dist}(x_0,\partial B_d)=d\gg h$), we have $\psi_0\in\mathring{S}_h^{k-1,\mu-1}(B_d)$.
Construct the corrected truncated delta function
\begin{equation*}
	\widetilde\delta_h:=T_h\delta_h - m\psi_0.
\end{equation*}
Evidently $\widetilde\delta_h$ differs from $T_h\delta_h$ only by the $O(h^M)$ perturbation $m\psi_0$, and $\int_{B_d}\widetilde\delta_h = m - m\int\psi_0 = 0$.
Moreover $\widetilde\delta_h\in\mathring{S}_h^{k-1,\mu-1}(B_d)$ because both $T_h\delta_h$ and $\psi_0$
belong to that space.  Let $\chi(x)=\int_{x_L}^{x}\widetilde\delta_h(t)\,\dd t$, where
$x_L$ is the left endpoint of $B_d$.  Since $\widetilde\delta_h$ vanishes at
$\partial B_d$ together with its derivatives up to order $\mu-1$ and has zero
integral over $B_d$, the primitive $\chi$ belongs to $\mathring{S}_h^{k,\mu}(B_d)$
and satisfies $\chi'=\widetilde\delta_h$.  Hence $\widetilde\delta_h\in W_h$.

Now, using $T_h\delta_h = \widetilde\delta_h + m\psi_0$, write
\[
(\delta_h,w) = (T_h\delta_h,w) + (\delta_h-T_h\delta_h,w)
= (m\psi_0,w) + (\widetilde\delta_h,w) + (\delta_h-T_h\delta_h,w).
\]
By \eqref{eq: orthogonality,w} and the fact $\widetilde\delta_h\in W_h$,  we have 
\[
(\widetilde\delta_h,w) = 0.
\] 
Consequently,
\[
(\Pi_{B_d}u-u_h)^{(s)}(x_0) = (m\psi_0,w) + (\delta_h-T_h\delta_h,w).
\]
Both terms are of arbitrarily high order:
\begin{align*}
|(m\psi_0,w)| &\le |m|\,\|\psi_0\|_{L^2(\Omega)}\,\|w\|_{L^2(\Omega)} \le C h^{M+k-1/2},\\
|(\delta_h-T_h\delta_h,w)| &\le \|\delta_h-T_h\delta_h\|_{L^2(\Omega)}\,\|w\|_{L^2(\Omega)} \le C h^{M+k},
\end{align*}
where we have used \eqref{eq: arbitray low order of m}, Lemma~\ref{lem:trunc} and the a priori estimate 
\[
\|w\|_{L^2(\Omega)}\leq |u-u_h|_{H^1(\Omega)}+|u- \Pi^*_{\Omega}u|_{H^1(\Omega)}\leq Ch^k. 
\]
This completes the proof.
\end{proof}

\begin{remark}
The proof for $s\ge2$ crucially relies on the fact that $\int_\Omega\delta_h = 0$, which forces $m = O(h^M)$ and thereby renders both terms in the final estimate arbitrarily small.  For $s=1$, however, $\int_\Omega\delta_h = 1$ (by \eqref{eq:discrete_delta} with $\chi=1$), so $m = 1+O(h^M)=O(1)$ is not small.  Consequently, the argument above cannot establish a superconvergence
estimate for $s=1$.  In fact, $(\Pi_{B_d}u-u_h)'(x_0)=O(h^{k})$.  The lower-order case falls outside the scope of the present technique, and one instead follows the approach of Wahlbin~\cite[Chapter 1]{Wahlbin1996}, who analyzes $(u-u_h)^{(s)}(x_0)$ directly for $s=0,1$.
\end{remark}

% ---------------------------------------------------------------------------
% 3.2 Local approximation error
% ---------------------------------------------------------------------------
\subsection{Local approximation error and symmetry-induced superconvergence points}
\label{sec:local_approx}

In this subsection, we estimate the second term $(u-\Pi_{B_d}u)^{(s)}(x_0)$ in~\eqref{eq:decomposition}. The discussion is restricted to the case where $ x_0 $ serves as the symmetric center of $ B_d $, namely,  $ x_0 $ is either a meshpoint or a midpoint of the  mesh elements. In the end of this subsection, we combine the estimation of two terms  in~\eqref{eq:decomposition} to present the symmetry-induced superconvergence points. 

Define the reflection operator $ T $ such that 
\begin{equation*}
	Tf(x) = f(2x_0-x).
\end{equation*}
Let $ \{\chi_j\}_{j=1}^{N_0} $ be the set of basis functions in $ S_{h}^{k,\mu} $ whose support has nonempty intersection with $ B_d $. We require that $ x_0 $ be the symmetric center of a slightly larger region $ B_{d^*} $, where $ d^* $ is defined as the union of the supports of all elements in $ \{\chi_j\}_{j=1}^{N_0} $. Without confusion, let $ d^*=d $ since $ d^*= d+Ch\sim d $.
Provided $B_d(x_0)$ is symmetric about $x_0$, the reflection $T$ maps the spline space $S_h^{k,\mu}(B_d)$
onto itself.  
\begin{lemma}\label{lem:parity_local}
Assume $B_d(x_0)$ is symmetric about $x_0$.  If $u$ is an even (resp. odd)
function with respect to $x_0$, then the local spline projection
$\Pi_{B_d}u\in S_h^{k,\mu}(B_d)$ is also even (resp. odd) with respect
to $x_0$.
\end{lemma}
\begin{proof}
First we have $T\Pi_{B_d}u\in S_h^{k,\mu}(B_d)$.  Suppose $u$ is even, so
$Tu=u$.  For any test function $\chi\in\intr S_h^{k,\mu}(B_d)$, the
bilinear form satisfies
	\begin{equation*}
	a_{B_d}(u-T\Pi_{B_d}u,\chi)= a_{B_d}(T(u-\Pi_{B_d}u),\chi) = a_{B_d}(u-\Pi_{B_d}u,T\chi)=0,
\end{equation*}
Moreover, $u-T\Pi_{B_d}u = T(u-\Pi_{B_d}u)$ vanishes together with its derivatives
up to order $\mu$ at $\partial B_d$ (since $u-\Pi_{B_d}u$ does).  Thus
$T\Pi_{B_d}u$ satisfies both the orthogonality and the boundary conditions
defining $\Pi_{B_d}u$, and by uniqueness $T\Pi_{B_d}u=\Pi_{B_d}u$, which  means $\Pi_{B_d}u$ is even.

If $u$ is odd, then $Tu=-u$.  A similar argument gives that $\Pi_{B_d}u$ is odd.
\end{proof}
Based on Lemma~\ref{lem:parity_local}, we are now ready to present the following superconvergence result of the local approximation error $(u-\Pi_{B_d}u)^{(s)}(x_0)$ for $s\geq 2$.
\begin{lemma}\label{lemma,2}
	Assume $ x_0 $ is  the symmetric center of  $ B_d(x_0)$ with diameter $ d=C^*h|\ln h| $. If $k-s$ is even and $s\ge 2$, then
	\begin{equation}\label{lemma, u-pihu}
		|(u-\Pi_{B_d}u)^{(s)}(x_0)| \leq C h^{k+2-s}|u|_{W^{k+2,\infty}(B_d)},
	\end{equation}
	where $ \Pi_{B_d}u $ is the local polynomial spline solution satisfying \eqref{eq:local_projection}, and the constant $C$ is independent of $h$ and $\ln h$.
	Here, $s\le\mu$ when $x_0$ is a mesh node, and $s\le k$ when $x_0$ lies in the interior of an element.
\end{lemma}
\begin{proof}
Let $p_{k+1}$ be the Taylor polynomial of $u$ of degree $k+1$ at $x_0$ and write
$u=q_k+\alpha g+r$, where $q_k\in\mathcal{P}_k$, $\alpha=u^{(k+1)}(x_0)/(k+1)!$,
$g(x)=(x-x_0)^{k+1}$, and $r$ is the Taylor remainder of order $k+2$, so
that $r^{(j)}(x_0)=0$ for $j=0,\dots,k+1$.  Since $\Pi_{B_d}$ reproduces
all polynomials of degree at most $k$,
\begin{equation*}
(u-\Pi_{B_d}u)^{(s)}(x_0)
=(r-\Pi_{B_d}r)^{(s)}(x_0)
+\alpha\bigl(g-\Pi_{B_d}g\bigr)^{(s)}(x_0).
\end{equation*}
The function $g$ has parity $(-1)^{k+1}$ about $x_0$ and, by
Lemma~\ref{lem:parity_local}, so does $\Pi_{B_d}g$; hence
$\bigl(g-\Pi_{B_d}g\bigr)^{(s)}(x_0)=0$ whenever $k-s$ is even.
Moreover, $s\le k$, so $r^{(s)}(x_0)=0$ and
\begin{equation*}
|(u-\Pi_{B_d}u)^{(s)}(x_0)|=|(\Pi_{B_d}r)^{(s)}(x_0)|.
\end{equation*}
The estimate $|(\Pi_{B_d}r)^{(s)}(x_0)|\le C h^{k+2-s}|u|_{W^{k+2,\infty}(B_d)}$
for $s\ge2$ is proved in Appendix~\ref{app:lnh}
(Lemma~\ref{lem:app_sharp}); this completes the proof.
\end{proof}

The estimate \eqref{lemma, u-pihu} is sharp: it exhibits the full order
$k+2-s$ without any logarithmic factor.  In the proof above it is
reduced to the estimate of $(\Pi_{B_d}r)^{(s)}(x_0)$; the latter is
established in Appendix~\ref{app:lnh} by a somewhat lengthy and
technical argument based on local discrete delta functions.

Together with the projection difference estimate
(Lemma~\ref{lem:proj_diff_sge2}), we obtain the superconvergence of
$e^{(s)}(x_0)= (u-u_h)^{(s)}(x_0)$ for $s\ge2$.  For the remaining cases
$s=0,1$, Wahlbin~\cite[Chapter 1]{Wahlbin1996} directly analyzes
$(u-u_h)^{(s)}(x_0)$ and obtains the same enhanced rate.  Consequently,
we have the following fundamental result.

\begin{theorem}\label{thm:local_super}
Let $x_0\in\Omega$ be such that $B_d(x_0)$ is symmetric about $x_0$ with
$d = C^*h|\ln h|$ and $C^*$ sufficiently large.  If $k-s$ is even, then
\begin{equation*}
|(u-u_h)^{(s)}(x_0)| \le C h^{k+2-s},
\label{eq:local_super_result}
\end{equation*}
where $C$ is independent of $h$ and $\ln h$.  Here, $0\le s\le\mu$ when $x_0$ is a mesh node, and $0\le s\le k$ when $x_0$ lies in the interior of an element.
\end{theorem}

\begin{proof}
Decompose $e^{(s)}(x_0)=(\Pi_{B_d}u-u_h)^{(s)}(x_0)+(u-\Pi_{B_d}u)^{(s)}(x_0)$
as in \eqref{eq:decomposition}.  For $s\ge2$, the projection difference
is $O(h^M)$ for any $M>0$ by Lemma~\ref{lem:proj_diff_sge2}, and the
local approximation error satisfies
$|(u-\Pi_{B_d}u)^{(s)}(x_0)|\le C h^{k+2-s}$ when $k-s$ is even by
Lemma~\ref{lemma,2}.  For
$s=0,1$, the result is established in~\cite[Chapter 1]{Wahlbin1996}.
\end{proof}

% ===========================================================================
% 4. Asymptotic expansion on an element
% ===========================================================================
\section{Asymptotic expansion of the error on an element}
\label{sec:asymptotic}

Let $x_0\in\Omega$ be a fixed point such that the mesh is uniform within
$B_d(x_0)$ with $d=C^*h|\ln h|$ and $C^*$ sufficiently large.
Consider elements $I_i=[x_i,x_{i+1}]$ whose midpoints
$x_{i+1/2}$ satisfy $|x_{i+1/2}-x_0| = O(h)$.  These elements lie well
inside $B_d(x_0)$; consequently, Theorem~\ref{thm:local_super} applies at
their midpoints (for derivatives up to order $k$) and at their endpoints (for
derivatives up to order $\mu$).  In this section, we develop an asymptotic
expansion of the error $e=u-u_h$ on each such element, and determine the
leading-order coefficients from the combined Galerkin orthogonality and the superconvergence conditions.

First, we present the expansion of the error and preliminarily establish some high-order properties of the coefficients.
Map $I_i=[x_i,x_{i+1}]$ to $[-1,1]$ via
\begin{equation}
t=\frac{2(x-x_{i+1/2})}{h}\in[-1,1],\qquad
x_{i+1/2}=x_i+\frac{h}{2}.
\label{eq:affine_map}
\end{equation}
Expand the error in Legendre polynomials:
\begin{equation}\label{eq:legendre_expansion}
e|_{I_i}(t)=\sum_{j=0}^{k+1}c_{j,i}L_j(t)+R_i(x),
\end{equation}
where $L_j$ is the Legendre polynomial of degree $j$. 
By the Bramble--Hilbert lemma, the remainder satisfies $|R_i^{(s)}(x)|\le C h^{k+2-s}|u|_{W^{k+2,\infty}(I_i)}$
for $0\le s\le \mu$ if $x$ is an endpoint of $I_i$, and $0\le s\le k$ otherwise. The coefficients are
\[
c_{j,i}= \frac{2j+1}{2}\int_{-1}^1 e|_{I_i}(t)L_j(t)\,\dd t,\qquad
0\le j\le k+1,
\]
which satisfies $c_{j,i}=O(h^{k+1})$ generically. Using the superconvergence property  of the midpoints in Theorem~\ref{thm:local_super}, the following lemma shows that half of the coefficients are
of higher order.

\begin{lemma}\label{lemma:midpoint}
If $j\equiv k\pmod{2}$, then $c_{j,i}=O(h^{k+2})$.
\end{lemma}

\begin{proof}
By the parity of Legendre polynomials, 
\begin{equation}\label{eq: parity Lj}
	L_j^{(s)}(0)=0,\quad j-s  \text{ odd}.
\end{equation}
Differentiating~\eqref{eq:legendre_expansion} at the midpoint $t=0$ gives
\begin{equation}
e^{(s)}(x_{i+1/2}) = \frac{2^s}{h^s}\sum_{j=s}^{k+1} c_{j,i} L_j^{(s)}(0) + R_i^{(s)}(x_{i+1/2}).
\label{eq:midpoint_expansion}
\end{equation}
Take $s=k$ in~\eqref{eq:midpoint_expansion}.  Theorem~\ref{thm:local_super}
with $s=k$ (since $k-k=0$ is even) gives $|e^{(k)}(x_{i+1/2})|\le Ch^2$.
Hence, by \eqref{eq: parity Lj},
\[
|\frac{2^k}{h^k}c_{k,i}L_{k}^{(k)}(0)| \le Ch^2.
\]
which implies $c_{k,i}=O(h^{k+2})$.
Then take $s=k-2,k-4,\dots$ successively.  For each such $s$, $k-s$ is even,
so Theorem~\ref{thm:local_super} yields $|e^{(s)}(x_{i+1/2})|\le Ch^{k+2-s}$. By \eqref{eq: parity Lj},
the sum in~\eqref{eq:midpoint_expansion} involves only indices $j\equiv k\pmod{2}$.
All $c_{j,i}$ with $j>s$ and $j\equiv k\pmod{2}$ are already $O(h^{k+2})$ from
previous steps, thus they contribute $O(h^{k+2-s})$.  Consequently,
\[
|\frac{2^s}{h^s}\,c_{s,i}L_s^{(s)}(0)| \le Ch^{k+2-s},
\]
which implies $c_{s,i}=O(h^{k+2})$.  As the indices $s=k,k-2,k-4,\dots$
exhaust all $j\le k$ with $j\equiv k\pmod{2}$, the lemma follows.
\end{proof}

Furthermore, the coefficients $c_{j,i}$ are essentially independent of the
element index.  By the translation invariance of the mesh within $B_d$, we
have the following result.

\begin{lemma}\label{lemma:const_mode}
There exist numbers
$c_j$ (independent of $i$) such that, for every element $I_i$ with midpoint
$x_{i+1/2}$ satisfying $|x_{i+1/2}-x_0| = O(h)$,
\[
c_{j,i}=c_j+\delta_{j,i},\qquad |\delta_{j,i}|\le C h^{k+2},
\]
where $C$ depends on $k$, $\mu$, and $u$, but not on $h$ or $i$.
\end{lemma}
\begin{proof}
It suffices to prove the following inequality for every pair of adjacent
elements $I_\ell$ and $I_{\ell+1}$ with midpoints $x_{\ell+1/2},x_{\ell+3/2}$:
\begin{equation}\label{eq:const_diff}
|c_{j,\ell+1}-c_{j,\ell}| \le C_0\,h^{k+2},
\end{equation}
for then the triangle inequality gives the desired result.

\noindent\textit{Case 1: $j\equiv k\pmod{2}$.}
By Lemma~\ref{lemma:midpoint}, $c_{j,\ell},c_{j,\ell+1}=O(h^{k+2})$.
Hence their difference is also $O(h^{k+2})$, and~\eqref{eq:const_diff} holds
immediately.

\noindent\textit{Case 2: $j\ge 2$ and $j\not\equiv k\pmod{2}$.}
Let $\Pi_{\ell}$ be the local Ritz projection on
$B_d(x_{\ell+1/2})$ (cf.\ \eqref{eq:local_projection}).
Decompose $e=(u-\Pi_{\ell}u)+(\Pi_{\ell}u-u_h)$. We analyze each term on $I_\ell$.

Taylor-expand $u$ about $x_{\ell+1/2}$:
\[
u(x)=\sum_{m=0}^{k}\frac{u^{(m)}(x_{\ell+1/2})}{m!}(x-x_{\ell+1/2})^m
+\frac{u^{(k+1)}(x_{\ell+1/2})}{(k+1)!}(x-x_{\ell+1/2})^{k+1}
+R_{u,l}(x),
\]
where $|R_{u,l}(x)|\le C|x-x_{\ell+1/2}|^{k+2}$.
Since $\Pi_{\ell}$ reproduces all polynomials of degree $\le k$, we have
\[
(u-\Pi_{\ell}u)(x)
=\frac{u^{(k+1)}(x_{\ell+1/2})}{(k+1)!}
 \Bigl((x-x_{\ell+1/2})^{k+1}-\Pi_{\ell}
[(x-x_{\ell+1/2})^{k+1}](x)\Bigr)
+(R_{u,l}-\Pi_{\ell}R_{u,l})(x).
 \]
Map $I_\ell$ to $[-1,1]$ via \eqref{eq:affine_map}.
On this reference element, $(x-x_{\ell+1/2})^{k+1}=(\tfrac{h}{2})^{k+1}t^{k+1}$.
Let $Q_{k,\ell}$ be the polynomial such that
\[
\Pi_{\ell}[(x-x_{\ell+1/2})^{k+1}](x_{\ell+1/2}+\tfrac{h}{2}t)
= \bigl(\tfrac{h}{2}\bigr)^{k+1} Q_{k,\ell}(t).
\]
Since $|x_{\ell+1/2}-x_0|=O(h)$ and $d=C^*h|\ln h|$, the domain
$B_d(x_{\ell+1/2})$ is contained in $B_{d+Ch}(x_0)$.
The mesh is uniform inside $B_d(x_0)$, hence also
uniform on $B_d(x_{\ell+1/2})$.  Therefore the spline spaces
$S_h^{k,\mu}(B_d(x_{\ell+1/2}))$ for different $\ell$ are
translations of one another.  Let $T_\ell$ be the translation operator
$(T_\ell f)(x)=f(x+x_{\ell+1/2}-x_0)$.
Because the bilinear form $a_{B_d}(\cdot,\cdot)$ and the boundary conditions
in~\eqref{eq:local_projection} are invariant under translation of the
domain when the mesh is uniform, the local Ritz projection satisfies
$T_\ell\circ\Pi_{B_d(x_{\ell+1/2})} = \Pi_{B_d(x_0)}\circ T_\ell$.
Consequently $Q_{k,\ell}\equiv Q_k\in\mathcal{P}_k$, independent of $\ell$.  Then on $I_\ell$
\[
(x-x_{\ell+1/2})^{k+1}-
\Pi_{\ell}[(x-x_{\ell+1/2})^{k+1}](x)
= h^{k+1}\psi_{k+1}(t),\qquad
\psi_{k+1}(t):=\frac{1}{2^{k+1}}\bigl(t^{k+1}-Q_k(t)\bigr).
\]
The remainder satisfies $R_{u,l}=O(h^{k+2})$ on $I_\ell$, and
by stability of $\Pi_{\ell}$, $\|R_{u,l}-\Pi_{\ell}R_{u,l}\|_{L^\infty}
\le Ch^{k+2}$.  Therefore
\[
(u-\Pi_{\ell}u)(x_{\ell+1/2}+\tfrac{h}{2}t)
=\frac{u^{(k+1)}(x_{\ell+1/2})}{(k+1)!}\,h^{k+1}\psi_{k+1}(t)
+O(h^{k+2}).
\]

Consider the projection difference $v:=\Pi_{\ell}u-u_h$ on $I_\ell$.
Lemma~\ref{lem:proj_diff_sge2} (centered at $x_{\ell+1/2}$) gives
$v^{(s)}(x_{\ell+1/2})=O(h^M)$ for all $s\ge 2$.
Since $v$ is a polynomial of degree $\le k$ on $I_\ell$, Taylor expansion
about $x_{\ell+1/2}$ yields
\[
v\bigl(x_{\ell+1/2}+\tfrac{h}{2}t\bigr)
= v(x_{\ell+1/2}) + v'(x_{\ell+1/2})\tfrac{h}{2}t
+ \sum_{s=2}^{k} \frac{v^{(s)}(x_{\ell+1/2})}{s!}\bigl(\tfrac{h}{2}\bigr)^{s}t^{s}
= v(x_{\ell+1/2}) + v'(x_{\ell+1/2})\tfrac{h}{2}t+ O(h^{M+2}),
\]
where the $O(h^{M+2})$ bound collects the terms $s\ge 2$.
For $j\ge 2$, $L_j$ is $L^2(-1,1)$-orthogonal to $\{1,t\}$, so
\[
\int_{-1}^{1} v(t)\,L_j(t)\,\dd t = O(h^{M+2}).
\]
Choosing $M\ge k$ makes this contribution $O(h^{k+2})$, hence negligible.
Consequently
\[
c_{j,\ell}=
\frac{2j+1}{2}\int_{-1}^1 (u-\Pi_{\ell}u)(t)L_j(t)\,\dd t
+O(h^M)
=\beta_j\,\frac{u^{(k+1)}(x_{\ell+1/2})}{(k+1)!}\,h^{k+1}+O(h^{k+2}),
\]
where $\beta_j=\frac{2j+1}{2}\int_{-1}^1\psi_{k+1}(t)L_j(t)\,\dd t$.
Applying the same argument to $I_{\ell+1}$ and using
$|u^{(k+1)}(x_{\ell+3/2})-u^{(k+1)}(x_{\ell+1/2})|
= h\,u^{(k+2)}(\xi)=O(h)$ (which requires
$u\in W^{k+2,\infty}(B_d)$, as assumed throughout),
\[
|c_{j,\ell+1}-c_{j,\ell}|
=|\beta_j|\,\frac{|u^{(k+1)}(x_{\ell+3/2})-u^{(k+1)}(x_{\ell+1/2})|}
{(k+1)!}\,h^{k+1}+O(h^{k+2})=O(h^{k+2}),
\]
proving~\eqref{eq:const_diff} for $j\ge2$.

\noindent\textit{Case 3: $j=0$ and $k$ odd.}
At the node $x_{\ell+1}$ between $I_{\ell}$ and $I_{\ell+1}$, by the continuity of $e$,  evaluating the Legendre expansion at the node ($t=1$ for $I_\ell$,
$t=-1$ for $I_{\ell+1}$) and using $L_j(1)=1$, $L_j(-1)=(-1)^j$ yields
\[
\sum_{j=0}^{k+1}c_{j,\ell}=\sum_{j=0}^{k+1}(-1)^jc_{j,\ell+1}+O(h^{k+2}).
\]
For odd~$k$, Lemma~\ref{lemma:midpoint} gives $c_{j,\ell},c_{j,\ell+1}=O(h^{k+2})$ for all odd~$j$.
For even $j\ge2$, Case~2 gives $c_{j,\ell}=c_{j,\ell+1}+O(h^{k+2})$. The remaining terms give
	\[
	c_{0,\ell}
	= c_{0,\ell+1}+O(h^{k+2}).
	\]
	
	\noindent\textit{Case 4: $j=1$ and $k$ even.}
	For even $k$, Theorem~\ref{thm:local_super} with $s=0$ gives
	$e(x_{\ell +1})=O(h^{k+2})$.
	Evaluating the Legendre expansion at the node ($t=1$ for $I_\ell$,
	$t=-1$ for $I_{\ell+1}$) yields
	\[
	\sum_{j=0}^{k+1} c_{j,\ell}+
	\sum_{j=0}^{k+1} (-1)^j c_{j,\ell+1}=O(h^{k+2}).
	\]
	By Lemma~\ref{lemma:midpoint}, all even-indexed coefficients
	are $O(h^{k+2})$.  Moreover, odd $j\ge3$ satisfy
	$c_{j,\ell}=c_{j,\ell+1}+O(h^{k+2})$ by Case~2.
	The remaining terms give
\[
c_{1,\ell}
-c_{1,\ell+1}=O(h^{k+2}).
\]

Thus~\eqref{eq:const_diff} holds for all $j$, and the lemma is proved.
\end{proof}

Therefore, by Lemma~\ref{lemma:midpoint} and \ref{lemma:const_mode}, only
element-independent coefficients with parity opposite to $k$ contribute at
leading order, i.e.\ the index set
\begin{equation}
\mathcal{U}_k:=\{j=0,1,2,\dots,k+1: j\equiv k+1\pmod{2}\},
\qquad
|\mathcal{U}_k|=
\begin{cases}
\dfrac{k}{2}+1, & k\text{ even},\\[2mm]
\dfrac{k+1}{2}+1, & k\text{ odd}.
\end{cases}
\label{eq:U_set}
\end{equation}
These coefficients $\{c_j\}_{j\in \mathcal{U}_k}$ are determined by the Galerkin orthogonality and nodal
superconvergence constraints, which we develop in
Subsection~\ref{sec:determination}.  The special case $\mu=k-1$
(smoothest B-splines) admits an explicit closed form via an antiderivative
operator and is treated separately in Subsection~\ref{sec:bspline}.

% ---------------------------------------------------------------------------
% 4.1 Determination of the leading-order coefficients
% ---------------------------------------------------------------------------
\subsection{Determination of the leading-order coefficients $\{c_j\}_{j\in \mathcal{U}_k}$}
\label{sec:determination}

% ---------------------------------------------------------------------------
% 4.1.1 Galerkin orthogonality constraints
% ---------------------------------------------------------------------------
\subsubsection{Galerkin orthogonality constraints}
\label{sec:galerkin_constraints}
The Galerkin orthogonality \eqref{eq:FEM} implies
\begin{equation}
	(e',\chi)_{L^2(\Omega)}=0\qquad\forall\,\chi\in\widetilde W_h,
	\label{eq:galerkin_W}
\end{equation}
where
\begin{equation*}
	\widetilde W_h:=\bigl\{\chi\in \mathring{S}_h^{k-1,\mu-1}(\Omega): \textstyle\int_\Omega\chi=0\bigr\}.
	\label{eq:Wt_def}
\end{equation*}
Indeed, for any $\chi\in\widetilde W_h$, there exists $\eta\in S_h^{k,\mu}\cap H_0^1(\Omega)$ such that $\eta'=\chi$.

\begin{lemma}\label{lem:Phi_patch}
	Let $P = \bigcup_{\ell=1}^{M} I_{i+\ell-1}$ be a patch of $M\ge k$ consecutive elements, and let $\chi\in\widetilde W_h$ be supported on $P$.  For each $\ell=1,\dots,M$, map $I_{i+\ell-1}$ to the reference interval $[-1,1]$ via \eqref{eq:affine_map} and define the pullback polynomial
	\[
	\chi_\ell(t) := \chi\bigl(x_{i+\ell-1/2}+\tfrac{h}{2}t\bigr)\in\mathcal{P}_{k-1}([-1,1]).
	\]
	Then the summed reference polynomial
	\[
	\Phi(t) := \sum_{\ell=1}^{M} \chi_\ell(t)\in\mathcal{P}_{k-1}([-1,1])
	\]
	satisfies: 
	\[
	\displaystyle\int_{-1}^1\Phi(t)\,dt = 0, \qquad \Phi^{(m)}(-1) = \Phi^{(m)}(1) \text{ for } m=0,1,\dots,\mu-1.
	\]
	Conversely, any polynomial $\Phi\in\mathcal{P}_{k-1}$ satisfying the above conditions can be realized by some $\chi\in\widetilde W_h$ supported on a suitable patch.  Hence the set of admissible $\Phi$ is precisely
	\[
	V := \bigl\{\Phi\in\mathcal{P}_{k-1}([-1,1]): \Phi^{(m)}(-1)=\Phi^{(m)}(1),\; m=0,\dots,\mu-1,\; \textstyle\int_{-1}^1\Phi=0\bigr\},
	\]
	and $\dim V = k-\mu-1$.
\end{lemma}
\begin{proof}
	The zero-mean condition follows directly:
	\[
	0 = \int_\Omega\chi = \sum_{\ell=1}^{M}\int_{I_{i+\ell-1}}\chi
	= \frac{h}{2}\sum_{\ell=1}^{M}\int_{-1}^1\chi_\ell
	= \frac{h}{2}\int_{-1}^1\Phi.
	\]
	For the periodic-like conditions, note that $\chi$ is $C^{\mu-1}$, so at each internal node $x_{i+\ell}$ ($\ell=1,\dots,M-1$),
	\[
	\chi_\ell^{(m)}(1) = \chi_{\ell+1}^{(m)}(-1),\qquad m=0,\dots,\mu-1. 
	\]
	Since $\operatorname{supp}\chi\subset P$, $\chi$ and all derivatives vanish at the patch endpoints $x_i$ and $x_{i+M}$, giving
	\[
	\chi_1^{(m)}(-1)=0,\qquad \chi_M^{(m)}(1)=0,\qquad m=0,\dots,\mu-1. 
	\]
		Now evaluate $\Phi^{(m)}$ at $\pm1$:
	\begin{align*}
			\Phi^{(m)}(-1) &= \sum_{\ell=1}^{M}\chi_\ell^{(m)}(-1)
			= \underbrace{\chi_1^{(m)}(-1)}_{=0}+\sum_{\ell=2}^{M}\chi_\ell^{(m)}(-1)
			= \sum_{\ell=2}^{M}\chi_{\ell-1}^{(m)}(1)
			= \sum_{\ell=1}^{M-1}\chi_\ell^{(m)}(1),\\
			\Phi^{(m)}(1) &= \sum_{\ell=1}^{M}\chi_\ell^{(m)}(1)
			= \sum_{\ell=1}^{M-1}\chi_\ell^{(m)}(1) + \underbrace{\chi_M^{(m)}(1)}_{=0}=\sum_{\ell=1}^{M-1}\chi_\ell^{(m)}(1).
		\end{align*}
		Hence $\Phi^{(m)}(-1)=\Phi^{(m)}(1)$.
	
		The dimension count follows because $\mathcal{P}_{k-1}$ has dimension $k$; the $\mu$ periodic conditions are linearly independent, and the zero-mean condition is independent of them (e.g., $\Phi\equiv1$ satisfies the periodic conditions but not the zero-mean condition).  Hence $\dim V = k - \mu - 1$.
	\end{proof}

Under Lemma~\ref{lemma:const_mode}, the pullback of $e$ to the reference
interval $[-1,1]$ is, to leading order,
\begin{equation*}
e(t)=\sum_{j\in\mathcal{U}_k}c_j L_j(t)+O(h^{k+2}),
\end{equation*}
The Galerkin orthogonality~\eqref{eq:galerkin_W} then implies
\begin{equation*}
\sum_{j\in\mathcal{U}_k}c_j \int_{-1}^1 L'_j(t)\Phi(t)\,\dd t=O(h^{k+2})\qquad\forall\,\Phi\in V.
\end{equation*}
Using the Legendre expansion
\begin{equation*}
L_j'(t) = \sum_{\substack{m=0\\ m\equiv j-1\pmod{2}}}^{j-1} (2m+1)\,L_m(t),
\end{equation*}
and exchanging the order of summation, a direct calculation gives
\begin{align*}
\sum_{j\in\mathcal{U}_k}c_j \int_{-1}^1 L'_j(t)\Phi(t)\,\dd t
&= \sum_{j\in\mathcal{U}_k}c_j
   \sum_{\substack{m=0\\ m\equiv j-1\pmod{2}}}^{j-1} (2m+1)
   \int_{-1}^1 L_m(t)\Phi(t)\,\dd t \nonumber\\
&= 2\sum_{\substack{m=0\\ m\equiv k\pmod{2}}}^{k}
      \Bigl(\sum_{\substack{j\ge m+1\\ j\in\mathcal{U}_k}} c_j\Bigr)\Phi_m,
\end{align*}
where $\Phi_m=\frac{2m+1}{2}\int_{-1}^1\Phi(t)L_m(t)\,\dd t$ are the
Legendre coefficients of $\Phi$. Define the partial sums
\begin{equation}
\alpha_m := \sum_{\substack{j\ge m+1\\ j\in\mathcal{U}_k}} c_j,  \qquad  0\le m\le k,\,\,\,
m\equiv k\pmod{2},
\label{eq:alpha_def}
\end{equation}
then we have
\begin{equation*}
\sum_{\substack{m=0\\ m\equiv k\pmod{2}}}^{k} \alpha_m\,\Phi_m
= O(h^{k+2})\qquad\forall\,\Phi\in V.
\end{equation*}
Note that $\Phi_k=0$ since $\Phi\in\mathcal{P}_{k-1}$.  Hence $\alpha_k=c_{k+1}$ does not appear
in the Galerkin orthogonality sum.  The effective indices in the sum are therefore
$0\le m\le k-2$ with $m\equiv k\pmod{2}$.
For convenience, we drop the higher-order terms and work with the 
relation
	\begin{equation}\label{eq:orth_coeff}
	\sum_{\substack{m=0\\ m\equiv k\pmod{2}}}^{k-2} \alpha_m\,\Phi_m = 0 \qquad\forall\,\Phi\in V,
	\end{equation}
	which holds up to an $O(h^{k+2})$ remainder that does not affect
	the subsequent analysis.

The following lemma shows how many independent linear relations the
Galerkin orthogonality \eqref{eq:orth_coeff} imposes on the coefficients $\alpha_m$, and
consequently on the leading-order coefficients $\{c_j\}_{j\in\mathcal{U}_k}$.

\begin{lemma}[Galerkin orthogonality constraints]\label{lem:galerkin_count}
The Galerkin orthogonality~\eqref{eq:orth_coeff}
provides exactly
\begin{eqnarray}\label{eq,N,G}
	N_{\text{Gal}}=
	\begin{cases}
		\dfrac{k}{2} - \left\lfloor\dfrac{\mu}{2}\right\rfloor - 1, & k\text{ even},\\[6pt]
		\dfrac{k-1}{2} - \left\lfloor\dfrac{\mu+1}{2}\right\rfloor,  & k\text{ odd},
	\end{cases}
\end{eqnarray}
independent linear equations among the leading-order coefficients
$\{c_j\}_{j\in\mathcal{U}_k}$.
\end{lemma}

\begin{proof}
Equation~\eqref{eq:orth_coeff} states that the linear functional
\[
L(\Phi) = \sum_{\substack{m=0\\ m\equiv k\pmod{2}}}^{k-2} \alpha_m \Phi_m
\]
vanishes on $V$.  Hence $L\in V^\perp$.
Since $\dim V = k-\mu-1$, we have
$\dim V^\perp = \mu+1$.  The space $V^\perp$ is spanned by the
$\mu+1$ linear functionals that define the constraints on $V$:
\begin{align*}
\delta(\Phi) = \int_{-1}^1\Phi = 2\Phi_0, \qquad
\epsilon_r(\Phi) = \Phi^{(r)}(1)-\Phi^{(r)}(-1), \quad r=0,\dots,\mu-1.
\end{align*}
Here $\delta$ enforces the zero-mean condition, while $\epsilon_0,\dots,\epsilon_{\mu-1}$
enforce the periodic-like smoothness at the element boundaries. Therefore there
exist coefficients $\lambda_0,\dots,\lambda_{\mu-1}$ such that for all
$\widetilde\Phi\in\mathcal{P}_{k-1}([-1,1])$ and $ \int_{-1}^1\widetilde\Phi=0$,
\begin{equation}
L(\widetilde\Phi) = \sum_{\substack{m=1\\ m\equiv k\pmod{2}}}^{k-2} \alpha_m \widetilde\Phi_m
= \sum_{r=0}^{\mu-1}\lambda_{r}\bigl(\widetilde\Phi^{(r)}(1)-\widetilde\Phi^{(r)}(-1)\bigr).
\label{eq:galerkin_annihilator}
\end{equation}
Here $\widetilde\Phi_m$ denotes the Legendre coefficient of $\widetilde\Phi$ and we use $ \widetilde\Phi_0=\int_{-1}^1\widetilde\Phi=0$.
Expanding the right-hand side of~\eqref{eq:galerkin_annihilator} in the Legendre basis
$\widetilde\Phi(t)=\sum_{m=1}^{k-1}\widetilde\Phi_m L_m(t)$  gives
\begin{align*}
\sum_{r=0}^{\mu-1}\lambda_{r}\bigl(\widetilde\Phi^{(r)}(1)-\widetilde\Phi^{(r)}(-1)\bigr)
&= \sum_{r=0}^{\mu-1}\sum_{m=1}^{k-1}\lambda_{r}\widetilde\Phi_m L_m^{(r)}(1)\bigl(1-(-1)^{m-r}\bigr)\\
&= \sum_{m=1}^{k-1}\Bigl(\sum_{\substack{r=0\\ r\equiv m+1\pmod{2}}}^{\mu-1}2\lambda_{r} L_m^{(r)}(1)\Bigr)\widetilde\Phi_m.
\end{align*}
Comparing the coefficients of $\widetilde\Phi_m$ on both sides of~\eqref{eq:galerkin_annihilator}
gives two sets of equations: 
\begin{equation}
\alpha_m = \sum_{\substack{r=0\\ r\equiv m+1\pmod{2}}}^{\mu-1}2\lambda_{r} L_m^{(r)}(1),\quad 1\le m\le k-2,\,\,\,  m\equiv k\pmod{2},
\label{eq:alpha_lambda}
\end{equation}
and 
\begin{equation}
0 =  \sum_{\substack{r=0\\ r\equiv m+1\pmod{2}}}^{\mu-1}2\lambda_{r} L_m^{(r)}(1), 
\quad 1\le m\le k-1,\,\,\,  m\equiv k+1\pmod{2}.
\label{eq:lambda_constraints}
\end{equation}
Define the two sets of admissible $m$ (for each $k$)  and $r$ (for each $m$ and $\mu$) in \eqref{eq:alpha_lambda}:
\begin{equation*}
	\mathcal{M}_{k} = \{ m: \,1\le m\le k-2,\,  m\equiv k\pmod{2} \},
\end{equation*}
and,  since $ L_m^{(r)}(1)=0$ for $r>m$,
\begin{equation*}
   \mathcal{R}_{m,\mu} = \{ r: \, 0\le r\le \min\{\mu-1, m\},\,\,\,  r\equiv m+1\pmod{2} \}.
\end{equation*}
Equation~\eqref{eq:alpha_lambda} is a system of $\#\mathcal{M}_{k}$ equations
for $\alpha_m$ in terms of only $\#\mathcal{R}_{k-2,\mu}$ free parameters
$\lambda_{r}$.  The remaining
$\lambda_{r}$ are determined
independently by~\eqref{eq:lambda_constraints} and do not affect 
$\alpha_m$. 
The coefficient matrix $\{L_m^{(r)}(1)\}_{m\in\mathcal{M}_{k} ,r\in\mathcal{R}_{m,\mu}}$
  has full column rank $\#\mathcal{R}_{k-2,\mu}$. The $\{\alpha_m\}_{m\in\mathcal{M}_{k}}$  satisfy
$\#\mathcal{M}_{k}-\#\mathcal{R}_{k-2,\mu}$ independent linear relations, which translate to the same number of relations
among $\{c_j\}_{j\in\mathcal{U}_k}$.  Counting $\#\mathcal{M}_{k}$ and $\#\mathcal{R}_{k-2,\mu}$ yields:
\begin{itemize}
  \item \textbf{Even $k$}: $\#\mathcal{M}_{k}=k/2-1$ and $\#\mathcal{R}_{k-2,\mu}=\lfloor\mu/2\rfloor$.  Hence the first equation of \eqref{eq,N,G} follows.
  \item \textbf{Odd $k$}: $\#\mathcal{M}_{k}=(k-1)/2$ and $\#\mathcal{R}_{k-2,\mu}=\lfloor(\mu+1)/2\rfloor$.  Hence the second equation of \eqref{eq,N,G} follows.
\end{itemize}
This completes the proof.
\end{proof}

% ---------------------------------------------------------------------------
% 4.1.2 Nodal superconvergence constraints and closure theorem
% ---------------------------------------------------------------------------
\subsubsection{Nodal superconvergence constraints and closure theorem for $\{c_j\}_{j\in\mathcal{U}_k}$}
\label{sec:nodal_closure}

The following lemma shows how many independent linear relations
Theorem~\ref{thm:local_super} for nodal points imposes on the
leading-order coefficients $\{c_j\}_{j\in\mathcal{U}_k}$.

\begin{lemma}[Nodal superconvergence constraints]\label{lem:nodal_count}
The nodal superconvergence result
provides exactly
\[
N_{\text{Nod}} = \#\{s:\,0\le s\le\mu, s\equiv k\pmod{2}\}
= \begin{cases}
\left\lfloor\dfrac{\mu}{2}\right\rfloor+1, & k\text{ even},\\[8pt]
\left\lfloor\dfrac{\mu+1}{2}\right\rfloor, & k\text{ odd},
\end{cases}
\]
linear equations among the coefficients $\{c_j\}_{j\in\mathcal{U}_k}$.
\end{lemma}

\begin{proof}
Theorem~\ref{thm:local_super} gives $|e^{(s)}(x_i)|\le Ch^{k+2-s}$
whenever $k-s$ is even and $s\le\mu$ (at nodes $x_i$).
With the expansion \eqref{eq:legendre_expansion}, we have
\begin{equation*}
	\sum_{j\in\mathcal{U}_k} c_j L_j^{(s)}(1) = O(h^{k+2-s}),\qquad 0\le s\le\mu,\,\,\,
	s\equiv k\pmod{2}.
\end{equation*}
For convenience, we drop the higher-order terms and work with the 
relation
\begin{equation}\label{eq: superconvergence constraints}
	\sum_{j\in\mathcal{U}_k} c_j L_j^{(s)}(1) = 0,\qquad 0\le s\le\mu,\,\,\,
	s\equiv k\pmod{2}.
\end{equation}
which holds up to an $O(h^{k+2-s})$ remainder.
Counting  the number of equations above completes the proof.
\end{proof}

Before constructing the closure theorem from the two constraint lemmas (Lemma~\ref{lem:galerkin_count} and \ref{lem:nodal_count}),
we note a special coefficient $c_0$ for odd $k$, which represents the
element-wise mean of the error
$c_0 = \frac12\int_{-1}^1 e|_{I_i}(t)\,\dd t$.
It is not captured by either lemma's constraints.  Indeed, the Galerkin orthogonality
constraints in Lemma~\ref{lem:galerkin_count} involve the sums $\alpha_m = \sum_{j\ge m+1,\,j\in\mathcal{U}_k}c_j$
for $j\ge1$, so $c_0$ never appears.  The nodal superconvergence constraints 
in Lemma~\ref{lem:nodal_count} have $s\equiv k\pmod{2}$; for odd $k$ this forces
$s\ge1$,  so $c_0$ is absent there
as well.  The following lemma shows that for certain cases, $c_0$ can be of
higher order.

\begin{lemma}[$c_0$ superconvergence]\label{lem:c0_super}
Let $k$ be odd.  For $\mu\le1$, the element-mean coefficient
satisfies $c_0 = 0$ exactly.  For $\mu\ge2$, if the uniform mesh region
$B_d(x_0)$ has diameter $d=Ch^\sigma$ with $0\leq \sigma<1$, then $c_0$ is
of higher order $O(h^{k+1+\min\{\,1-\sigma,\,1/2\,\}})$.
\end{lemma}

\begin{proof}
Take any $B\in S_h^{k-2,\mu-2}(B_d)$ with compact support.
Since $B$ is a spline of degree $k-2$ and smoothness $\mu-2$, the equation
$\psi''=B$ on $\Omega$ together with the Dirichlet boundary conditions
$\psi(a)=\psi(b)=0$ determines a unique $\psi$:
\[
\psi(x)=\int_a^x\int_a^y B(z)\,dz\,dy
-\frac{x-a}{b-a}\int_a^b\int_a^y B(z)\,dz\,dy.
\]
Each integration raises the polynomial degree by one and the global
smoothness by one, so $\psi\in S_h^{k,\mu}$.  The boundary conditions
give $\psi\in H_0^1(\Omega)$, hence $\psi\in V_h$. Galerkin orthogonality \eqref{eq:FEM} gives 
\begin{equation*}\label{eq,B orthogonality}
	(e,B)=-(e',\psi') =0.
\end{equation*}

For $\mu=0,1$, take $B$ to be the characteristic function of a single
element~$I_i$ ($B=1$ on $I_i$, $0$ elsewhere).  Indeed, since $\mu-2\le -1$,
the space $S_h^{k-2,\mu-2}$ admits discontinuous functions.  Then $(e,B)=\int_{I_i}e= hc_0=0$, so $c_0=0$.

For $\mu\ge2$, the space $S_h^{k-2,\mu-2}$ does not admit discontinuous functions.
Note that $B_d(x_0)$ is a uniform mesh region with diameter $d=Ch^\sigma$,
$\sigma>0$, containing $N_0 = d/h = C/h^{1-\sigma}$ elements.
Construct $B\in S_h^{k-2,\mu-2}$ with $\operatorname{supp}B\subset B_d$,
$B=1$ on the interior of $B_d$, tapering smoothly to $0$ across the
$O(1)$ boundary elements (width $O(h)$).  Let $P_{\text{int}}$ denote the
set of $N_0$ elements where $B=1$, and $P_{\text{b}}$ the $O(1)$
boundary elements where $0\leq B\leq1$.  Then
\[
0 = \int_{B_d} eB =
\sum_{I_i \subset P_{\text{int}}}\int_{I_i} e
+ \sum_{I_i \subset P_{\text{b}}}\int_{I_i} eB .
\]
By Lemma~\ref{lemma:const_mode}, adjacent differences satisfy
$|c_{0,i+1}-c_{0,i}|\le Ch^{k+2}$.  Indexing elements  in $P_{\text{int}}$
from $i=1$ to $i=N_0$, we have $c_{0,i}=c_0+\delta_i$ with
$|\delta_i|\le C\,i\,h^{k+2}$.  On interior elements,
$\int_{I_i}e = hc_{0,i}= hc_0+ h\delta_i$.
On boundary elements, $\int_{I_i} eB = O(h^{k+2})$.  Thus
\[
0 = N_0 h c_0 + h\sum_{i=1}^{N_0}\delta_i + O(h^{k+2}).
\]
Since $|\sum_{i=1}^{N_0}\delta_i|\le C N_0^2 h^{k+2}$, the middle term is
$O(N_0^2 h^{k+3})$.  With $N_0=O(1/h^{1-\sigma})$ this is $O(h^{k+1+2\sigma})$.
Dividing by $N_0h = O(h^\sigma)$ gives
\[
|c_0| \le C\bigl(h^{k+1+\sigma} + h^{k+2-\sigma}\bigr),
\]
which is of higher order for $\sigma<1$.  In fact, if $\sigma\le \tfrac12$, the subregion
$B_{h^{1/2}}(x_0)\subset B_d$ (since $d=h^\sigma\ge C h^{1/2}$)
gives, with $\sigma=\tfrac12$, the improved bound
$|c_0|\le C h^{k+\frac32}$.  For $\sigma>\tfrac12$, the second term
$h^{k+2-\sigma}=h^{k+1+(1-\sigma)}$ dominates.  Combining both cases, we have
$|c_0| \le C h^{k+1+\min\{\,1-\sigma,\,1/2\,\}}$.
\end{proof}

Now, using the Galerkin orthogonality constraints (Lemma~\ref{lem:galerkin_count}), the nodal
superconvergence conditions (Lemma~\ref{lem:nodal_count}), and the $c_0$ superconvergence  for
odd~$k$ (Lemma~\ref{lem:c0_super}),  we are ready to construct the closure theorem for leading-order coefficients $\{c_j\}_{j\in\mathcal{U}_k}$.
\begin{theorem}[Closure theorem]\label{thm:closure}
Let $k\ge2$ and $0\le\mu\le k-1$.  The Galerkin orthogonality constraints   \eqref{eq:alpha_lambda} and nodal superconvergence constraints  \eqref{eq: superconvergence constraints} are  linearly independent. They uniquely determine the ratios $c_j/c_{k+1}$ ($j\in\mathcal{U}_k\setminus\{k+1\}$) for even $k$, and the ratios
$c_j/c_{k+1}$ ($j\in\mathcal{U}_k\setminus\{0, k+1\}$) for odd $k$, respectively. The error in these ratios is of the same order as the
remainder, i.e.\ full $O(h^{k+2})$.
\end{theorem}

\begin{proof}
We verify that the Galerkin orthogonality and nodal superconvergence constraint sets are linearly
independent as follows.  View each equation of   \eqref{eq:alpha_lambda} and \eqref{eq: superconvergence constraints} as a linear functional on the space
$\mathbb{R}^{|\mathcal{U}_k|}$ of leading-order coefficients $\{c_j\}_{j\in\mathcal{U}_k}$.
Let $\{G_i\}_{i=1}^{N_{\text{Gal}}}$ and
$\{N_i\}_{i=1}^{N_{\text{Nod}}}$ be the Galerkin orthogonality and nodal superconvergence functionals,
and suppose they satisfy a linear relation
\begin{eqnarray}
	\sum_{i=1}^{N_{\text{Gal}}} a_i G_i + \sum_{i=1}^{N_{\text{Nod}}} b_i N_i = 0 \qquad\text{(zero functional on }\mathbb{R}^{|\mathcal{U}_k|}\text{)}.
	\label{eq:lin_relation}
\end{eqnarray}
The goal here is to prove that the only solution is  $a_i=b_i=0$ for all $i$, i.e., the independence of the
combined system $\{G_i,N_i\}$.

Define $V_G\subset\mathbb{R}^{|\mathcal{U}_k|}$ as the subspace where the
Galerkin  orthogonality functionals vanish:
\[
V_G = \{\,c\in\mathbb{R}^{|\mathcal{U}_k|}\mid G_i(c)=0,\;
i=1,\dots,N_{\text{Gal}}\,\}.
\]
Then, the independence of $\{N_i|_{V_G}\}$ implies the independence of the
combined system $\{G_i,N_i\}$.
To see this, first restrict the relation \eqref{eq:lin_relation} to $V_G$.  Since the $G_i$
vanish on $V_G$ by definition, it reduces to
\begin{equation}
	\sum_{i=1}^{N_{\text{Nod}}} b_i\, N_i|_{V_G}=0.
	\label{eq:restricted_rel}
\end{equation}
If $\{N_i|_{V_G}\}$ is independent, then \eqref{eq:restricted_rel} forces
$b_i=0$ for all $i$.  Substituting $b_i=0$ back into \eqref{eq:lin_relation}
gives $\sum_i a_i G_i=0$ on the full space $\mathbb{R}^{|\mathcal{U}_k|}$;
the independence of $\{G_i\}$ (Lemma~\ref{lem:galerkin_count}) then
yields $a_i=0$.  Hence it suffices to verify the independence of
$\{N_i|_{V_G}\}$. 

From \eqref{eq:alpha_def}, the leading-order coefficients are recovered by
$c_{m+1}=\alpha_{m}-\alpha_{m+2}$ ($0\le m\le k-2$, $m\equiv k\pmod{2}$) and
$c_{k+1}=\alpha_{k}$.
Substituting these into the nodal superconvergence constraints \eqref{eq: superconvergence constraints} and rearranging gives
\begin{equation}\label{eq:N_i_alpha}
	N_i
	=\sum_{m\in\mathcal{M}_k}\alpha_m\bigl(L_{m+1}^{(s_i)}(1)-L_{m-1}^{(s_i)}(1)\bigr)
	+\alpha_k\bigl(L_{k+1}^{(s_i)}(1)-L_{k-1}^{(s_i)}(1)\bigr)
	+\begin{cases}
		\alpha_0L_{1}^{(s_i)}(1), & k\text{ even},\\
		0, & k\text{ odd},
	\end{cases}
\end{equation}
where $\mathcal{M}_k$ is as in Lemma~\ref{lem:galerkin_count} and $s_i$
($i=1,\dots,N_{\text{Nod}}$) are the derivative orders $0\le s\le\mu$ with
$s\equiv k\pmod{2}$, in increasing order.
Here $c_0$ does not enter the constraints: for odd $k$, $L_0^{(s_i)}(1)=0$
since $s_i\ge1$, and for even $k$, $0\notin\mathcal{U}_k$.

On $V_G$, the coefficients $\alpha_m$
($m\in\mathcal{M}_k$) are expressed in terms of the free parameters
$\lambda_r$, $r\in\mathcal{R}_{k,\mu}$, via \eqref{eq:alpha_lambda}, with $\mathcal{R}_{m,\mu}$ defined
in Lemma~\ref{lem:galerkin_count}.
The coefficient $\alpha_k=c_{k+1}$ does not occur in the Galerkin
orthogonality (cf.\ the statement preceding \eqref{eq:orth_coeff}) and is
therefore free. For even $k$, the coefficient $\alpha_0$ is free as well,
since it drops out of the Galerkin orthogonality because
$\Phi_0=\frac12\int_{-1}^1\Phi=0$ for $\Phi\in V$.
For odd $k$, the coefficient $c_0\in\mathcal{U}_k$ is one of the unknowns
under consideration, but it appears in none of the constraints; it is therefore a further free parameter.
Hence $\dim V_G=|\mathcal{R}_{k,\mu}|+2$, i.e.\ $\dim V_G=N_{\text{Nod}}+1$ for even $k$ and
$N_{\text{Nod}}+2$ for odd $k$.
By the rank--nullity theorem, the independence of the $N_{\text{Nod}}$
functionals $\{N_i|_{V_G}\}$ is equivalent to the statement that the
equations $N_i|_{V_G}=0$ ($i=1,\dots,N_{\text{Nod}}$) leave free only the
overall scale $\alpha_k$ (and, for odd $k$, also $c_0$).
Substituting \eqref{eq:alpha_lambda} into \eqref{eq:N_i_alpha} gives
\begin{equation}\label{eq:N_VG}
	N_i|_{V_G}
	=\sum_{m\in\mathcal{M}_k}\bigl(L_{m+1}^{(s_i)}(1)-L_{m-1}^{(s_i)}(1)\bigr)
	\sum_{r\in \mathcal{R}_{m,\mu}}2\lambda_rL_m^{(r)}(1)
	+\alpha_k\bigl(L_{k+1}^{(s_i)}(1)-L_{k-1}^{(s_i)}(1)\bigr)
	+\begin{cases}
		\alpha_0L_{1}^{(s_i)}(1), & k\text{ even},\\
		0, & k\text{ odd},
	\end{cases}.
\end{equation}
For even $k$, the equation with $s_1=0$ reads
$N_1|_{V_G}=\alpha_0L_1(1)=\alpha_0=0$ (all other terms vanish because
$L_j(1)=1$ for every $j$); for odd $k$ there is no $\alpha_0$-term.
Using the identity
$L_{m+1}^{(s)}(1)-L_{m-1}^{(s)}(1)=(2m+1)L_m^{(s-1)}(1)$ ($s\ge1$), the
remaining equations (all equations when $k$ is odd) take the form
\begin{equation*}
	A\{\lambda_r\}+\alpha_k \{b_i\}=0,
\end{equation*}
where
\begin{equation*}
	A=\Bigl\{\sum_{m\in\mathcal{M}_k}2(2m+1)L_m^{(s_i-1)}(1)L_m^{(r)}(1)\Bigr\}_{i,r},
	\qquad
	b_i=L_{k+1}^{(s_i)}(1)-L_{k-1}^{(s_i)}(1),
\end{equation*}
and both the row indices $s_i-1$ and the column indices $r$ range over the
same set $\mathcal{R}_{k,\mu}$; in particular $A$ is square.
The matrix $A$ is a symmetric Gram matrix and is invertible.
Consequently, the equations $N_i|_{V_G}=0$ leave free only the overall
scale $\alpha_k=c_{k+1}$ and, for odd $k$, also the coefficient $c_0$:
for even $k$, they force $\alpha_0=0$ and
$\{\lambda_r\}=-\alpha_kA^{-1}\{b_i\}$; for odd $k$, they force
$\{\lambda_r\}=-\alpha_kA^{-1}\{b_i\}$, while $c_0$ does not enter the
equations.
In view of the equivalence stated above, the functionals
$\{N_i|_{V_G}\}$ are linearly independent.

From the above, the combined system  \eqref{eq:alpha_lambda} and \eqref{eq: superconvergence constraints} therefore  has
$N_{\text{Gal}}+N_{\text{Nod}}$ independent equations.
Counting these equations from Lemmas~\ref{lem:galerkin_count}
and~\ref{lem:nodal_count}:
\begin{align*}
	k\text{ even:}\quad &N_{\text{Gal}} = \frac{k}{2}-\Bigl\lfloor\frac{\mu}{2}\Bigr\rfloor-1,\,\,\, N_{\text{Nod}} = \Bigl\lfloor\frac{\mu}{2}\Bigr\rfloor+1,\,\,\,
	N_{\text{Gal}}+N_{\text{Nod}} = \frac{k}{2}=|\mathcal{U}_k| -1,\\[4pt]
	k\text{ odd:}\quad &N_{\text{Gal}} = \frac{k-1}{2}-\Bigl\lfloor\frac{\mu+1}{2}\Bigr\rfloor, \,\,\, N_{\text{Nod}} = \Bigl\lfloor\frac{\mu+1}{2}\Bigr\rfloor,\,\,\,
	N_{\text{Gal}}+N_{\text{Nod}} = \frac{k-1}{2}=|\mathcal{U}_k| -2.
\end{align*}
Consequently,  the combined independent equations  \eqref{eq:alpha_lambda} and \eqref{eq: superconvergence constraints} uniquely determine the $|\mathcal{U}_k|-1$ ratios $c_j/c_{k+1}$ ($j\in\mathcal{U}_k\setminus\{k+1\}$)  for even~$k$,  and  the $|\mathcal{U}_k|-2$ ratios $c_j/c_{k+1}$ ($j\in\mathcal{U}_k\setminus\{0,k+1\}$)  for odd~$k$, respectively.
\end{proof}

% ===========================================================================
\subsection{The special case $\mu=k-1$ (smoothest B-splines)}
\label{sec:bspline}

In this subsection, we specialize to the smoothest B-spline case
$\mu=k-1$.  From Lemma~\ref{lem:galerkin_count}, the Galerkin orthogonality \eqref{eq:alpha_lambda}  imposes no local constraints
\begin{equation*}
N_{\text{Gal}} = \begin{cases}
	\dfrac{k}{2}-\Bigl\lfloor\dfrac{k-1}{2}\Bigr\rfloor-1 = 0, \quad \text{for even } k \\[3mm]
	 \dfrac{k-1}{2}-\Bigl\lfloor\dfrac{k}{2}\Bigr\rfloor =  0,\quad \text{for odd } k
\end{cases},
\end{equation*}
so the leading-order coefficient relationships all
arise from the nodal superconvergence constraints \eqref{eq: superconvergence constraints}.  We will show that the
leading-order error polynomial admits an explicit description via an
antiderivative operator.

\begin{definition}\label{def:F_operator}
Define $\Lop$ by
\begin{equation}
\Lop(L_1)(t) = \frac{1}{3} L_2(t),\qquad
\Lop(L_j)(t) = \frac{1}{2j+1}\bigl(L_{j+1}(t)-L_{j-1}(t)\bigr),\quad j\ge2,
\label{eq:F_def_ref}
\end{equation}
extended linearly to all polynomials spanned by $\{L_1,L_2,\dots\}$.  For $k\ge1$, denote the $k$-fold
iterate by $\Lop^{k}(L_1)$ and set $F_k(t) := \Lop^{k}(L_1)(t)$.
\end{definition}

Using the derivative recurrence $(2j+1)L_j = L_{j+1}' - L_{j-1}'$, one verifies that $\Lop$ acts as a formal antiderivative:
\[
\frac{d}{dt}\bigl(\Lop(p)(t)\bigr) = p(t),\qquad
\forall p\in\operatorname{span}\{L_1,L_2,\dots\}.
\]
Consequently, for $k\ge2$,
\begin{equation}
\frac{d}{dt}F_{k}(t) = F_{k-1}(t),\qquad
F_{k}^{(s)}(t) = F_{k-s}(t)\quad (s< k).
\label{eq:F_derivative_chain}
\end{equation}

\begin{lemma}\label{lem:F_basic}
For $k\ge2$, the polynomials $F_k = \Lop^{k}(L_1)$ satisfy:
\begin{itemize}
  \item $F_k(-t) = (-1)^{k+1}F_k(t)$ (parity $(-1)^{k+1}$).
    Hence $F_k$ is odd when $k$ is even, and even when $k$ is odd.
  \item $\displaystyle\int_{-1}^{1} F_k(t)\,\dd t = 0$ for all $k\ge1$.
  \item $F_k(1) = 0$ for even $k\ge2$.
\end{itemize}
\end{lemma}

\begin{proof}
To show the first result,
we proceed by induction.  For $k=1$, $F_1 = \frac13 L_2$ and $L_2$ is even,
so $F_1(-t)=F_1(t)=(-1)^{2}F_1(t)$.  Assume $F_k$ has parity $(-1)^{k+1}$.
From~\eqref{eq:F_def_ref}, $\Lop(L_j)$ is a linear combination of
$L_{j+1}$ and $L_{j-1}$; since $L_{j\pm1}(-t)=(-1)^{j\pm1}L_{j\pm1}(t)
= -(-1)^j L_{j\pm1}(t)$, the operator $\Lop$ maps a polynomial of parity
$(-1)^j$ to one of parity $-(-1)^j$, i.e.\ $\Lop$ flips parity.
Consequently $F_{k+1}=\Lop(F_k)$ has parity $-(-1)^{k+1}=(-1)^{(k+1)+1}$,
completing the induction.

The second result is clear, since for every $j\ge1$, $\int_{-1}^{1}L_j=0$, and $F_k$ is a linear combination of $\{L_1,L_2,...\}$.

The last result is proved as follows.
By \eqref{eq:F_derivative_chain},  integrating from $0$ to $1$ gives
\[
F_k(1)-F_k(0)=\int_0^1 F_{k-1}(t)\,\dd t.
\]
If $k$ is even, $F_k$ is odd, so $F_k(0)=0$.
 Because
$\int_{-1}^1 F_{k-1}=0$ (zero integral) and $F_{k-1}$ is even,
we have $2\int_0^1 F_{k-1}=0$, whence $\int_0^1 F_{k-1}=0$.
Thus $F_k(1)=0$.  
\end{proof}

We now show that the ratios  $c_j/c_{k+1}$ ($j\in\mathcal{U}_k\setminus\{k+1\}$) for even $k$ and the ratios
$c_j/c_{k+1}$ ($j\in\mathcal{U}_k\setminus\{0, k+1\}$) for odd $k$ 
coincide with those of the Legendre coefficients of $F_k = \Lop^{k}(L_1)$.

\begin{lemma}\label{lem:bspline_coeff}
Let $k\geq 2$ and $\mu=k-1$ (smoothest B-splines).  Write
\begin{equation*}
F_k(t)=\Lop^{k}(L_1)(t)=\sum_{j\in\mathcal{U}_k\setminus\{0\}} \gamma_j L_j(t).
\label{eq:Fk_expansion}
\end{equation*}
Then the leading-order coefficients $\{c_j\}_{j\in\mathcal{U}_k}$  for even $k$ and $\{c_j\}_{j\in\mathcal{U}_k\setminus\{0\}}$ for odd $k$  of
the error satisfy
\begin{equation*}
\frac{c_j}{c_{k+1}} = \frac{\gamma_j}{\gamma_{k+1}},
\qquad j\in\mathcal{U}_k\setminus\{0,k+1\},
\label{eq:bspline_ratio}
\end{equation*}
up to $O(h^{k+2})$ remainders. 
\end{lemma}

\begin{proof}
For the nodal superconvergence constraints \eqref{eq: superconvergence constraints},write $s = k-2, k-4, \dots$ in
decreasing order (these are all values with $s\equiv k\pmod{2}$ and
$0\le s\le k-2$; note $s=k-1$ does not satisfy the parity condition).
At $s=k-2$, only $j\ge k-2$ contribute to~\eqref{eq: superconvergence constraints}. Thus, 
\[
c_{k-1}L_{k-1}^{(k-2)}(1) + c_{k+1}L_{k+1}^{(k-2)}(1)=0,
\]
which determines $c_{k-1}/c_{k+1}$ uniquely.
At $s=k-4$, indices $j=k-3,k-1,k+1$ contribute.  Using the already
determined ratio $c_{k-1}/c_{k+1}$, this equation fixes
$c_{k-3}/c_{k+1}$.  Proceeding downward through $s=k-6,k-8,\dots$,
each step introduces one new coefficient and determines its ratio to
$c_{k+1}$ uniquely.  

	Now consider $F_k$.  By the antiderivative property~\eqref{eq:F_derivative_chain},
	$F_k^{(s)}(t)=F_{k-s}(t)$.  Lemma~\ref{lem:F_basic} gives
	$F_{k-s}(1)=0$ whenever $k-s\ge2$ is even.
For $s\equiv k\pmod{2}$ and $s\le k-2$, the quantity $k-s$ is even
and at least~$2$, so $F_k^{(s)}(1)=F_{k-s}(1)=0$.  Hence the
Legendre coefficients $\{\gamma_j\}$ satisfy the same homogeneous
system~\eqref{eq: superconvergence constraints} as $\{c_j\}$. 
Thus,  $\{c_j\}$ and $\{\gamma_j\}$ are
proportional.  
\end{proof}

The polynomials $F_k$ possess a rich zero structure (i.e. the superconvergence points of B-splines) as follows.

\begin{lemma}\label{lem:F_zeros}
For $k\ge 2$, the zeros of $F_k(t)$ satisfy:
\begin{itemize}
  \item If $k$ is even, $F_k(t)$  has simple zeros at
    $t=-1,0,1$ and no other zeros in $[-1,1]$. If $k$ is odd, $F_k(t)$  has exactly two simple
    symmetric zeros $\pm a_k\in(-1,1)$ and no other zeros in $[-1,1]$.
  \item $\displaystyle\lim_{k\to\infty} a_k = \frac12$.  More precisely,
    \[
    a_k-\frac12 = \frac{1}{\pi 2^{k+1}} - \frac{1}{\pi 4^{k+1}}
    + \frac{4}{\pi 6^{k+1}} - \frac{17}{6\pi 8^{k+1}}
    + O\!\left(\frac{1}{10^{k+1}}\right).
    \]
\end{itemize}
\end{lemma}

\begin{proof}
We first prove the zero distribution by induction on $k$. For $k = 2$, the distribution of zeros $-1,0,1$ is clear. Assume the first statement of lemma holds for some $\tilde k \ge 2$.
If $\tilde k$ is even,  $F_{\tilde k}$ is odd with zeros exactly at $-1, 0, 1$, and does not change sign in $(0,1)$. Since $F_{\tilde k+1}' = F_{\tilde k}$, $F_{\tilde k+1}$ is monotonic on $[0,1]$. From Lemma~\ref{lem:F_basic}, the integral condition 
\begin{eqnarray*}
	\int_{-1}^1 F_{\tilde k+1}(x)\,\mathrm{d}x = 2\int_0^1 F_{\tilde k+1}(x)\,\mathrm{d}x = 0,
\end{eqnarray*}
forces $F_{\tilde k+1}$ to change sign in $(0,1)$. By monotonicity, there is exactly one zero in $(0,1)$, symmetric to one in $(-1,0)$. If $F_{\tilde k+1}(1) = 0$, then by monotonicity $F_{\tilde k+1}$ does not change sign on $(0,1)$, contradicting the integral condition. Thus $F_{\tilde k+1}(1) \neq 0$, and similarly $F_{\tilde k+1}(-1) \neq 0$.
If $\tilde k$ is odd,  $F_{\tilde k}$ is even with two symmetric zeros $\pm a_{\tilde k} \in (-1,1)$, and does not change sign in $(0,a_{\tilde k})$ and $(a_{\tilde k},1)$. By 
\begin{align*}
		F_{ \tilde k+1} = \Lop(F_{\tilde k}) =&  \Lop(C_{\tilde k+1}L_{\tilde k+1}+C_{\tilde k-1}L_{\tilde k-1}+...+C_2L_2)\\
		=& \frac{C_{\tilde k+1}}{2\tilde k+3}(L_{\tilde k+2}-L_{\tilde k})+\frac{C_{\tilde k-1}}{2\tilde k-1}(L_{\tilde k}-L_{\tilde k-2})+...+\frac{C_2}{5}(L_3-L_1),
\end{align*}
we deduce that $-1,0,1$ are zeros of  $F_{\tilde k+1} $, which has no other zeros (by monotonicity). This completes the induction. 

We prove $a_k\to\tfrac12$ and derive the precise asymptotics.
Consider the generating function
\begin{equation*}\label{eq:gen_func}
H(t,x) = \sum_{k=0}^\infty F_k(t)\,x^k,
\end{equation*}
with $F_0(t):=L_1(t)$.  By the antiderivative property~\eqref{eq:F_derivative_chain},
$F_k' = F_{k-1}$ for $k\ge1$, so
\begin{equation*}
\frac{\partial H}{\partial t}(t,x) = \sum_{k=1}^\infty F_{k-1}(t)\,x^k = x\,H(t,x) + 1.
\end{equation*}
Solving this ODE with $\int_{-1}^1 H(t,x)\,\mathrm{d}t=0$ (since
$\int_{-1}^1F_k=0$ for $k\ge1$) yields
\begin{equation*}
H(t,x) = \frac{e^{xt}}{\sinh x} - \frac{1}{x}.
\end{equation*}
Fixing $t$ and viewing $H(t,x)$ as a function of $x$, the term $e^{xt}/\sinh x$ has
simple poles at $x=n\pi i$ ($n\in\mathbb Z\setminus\{0\}$).  The residue is
\[
\lim_{x\to n\pi i}(x-n\pi i)\frac{e^{xt}}{\sinh x}
= \frac{e^{n\pi i t}}{\cosh(n\pi i)} = (-1)^n e^{n\pi i t},
\]
hence $H(t,x)$ admits the partial fraction expansion
\begin{equation*}
H(t,x) = \sum_{n\neq 0} \frac{(-1)^n e^{n\pi i t}}{x - n\pi i}.
	\end{equation*}
Expanding $1/(x-n\pi i) = -\sum_{k=0}^\infty x^{k}/(n\pi i)^{k+1}$ for
$|x|<\pi$ and extracting the coefficient of $x^k$ gives
\begin{align*}
F_k(t) &= \sum_{n\neq 0} (-1)^{n+1}\frac{e^{n\pi i t}}{(n\pi i)^{k+1}} \nonumber\\
&=- \frac{2(-1)^{(k+1)/2}}{\pi^{k+1}}
   \sum_{n=1}^{\infty}\frac{(-1)^{n+1}\cos(n\pi t)}{n^{k+1}},
\qquad k\text{ odd}.
\label{eq:Fourier}
\end{align*}
The series converges absolutely and uniformly on $[-1,1]$.
For odd~$k$, the positive zero $a_k$ satisfies $F_k(a_k)=0$, so
\begin{equation*}
\sum_{n=1}^{\infty}\frac{(-1)^{n+1}\cos(n\pi a_k)}{n^{k+1}}=0.
\end{equation*}
Set $\delta_k:=a_k-\tfrac12$.  Substituting $a_k=\tfrac12+\delta_k$ gives
\begin{equation}\label{eq:delta_eq}
\sin(\pi\delta_k)=	\sum_{m=1}^\infty \frac{(-1)^{m+1}\cos(2m\pi\delta_k)}{(2m)^{k+1}}
+ \sum_{m=2}^\infty \frac{(-1)^{m}\sin((2m-1)\pi\delta_k)}{(2m-1)^{k+1}}.
	\end{equation}
	A crude bound $|\sin(\pi\delta_k)|\le\sum_{n=2}^\infty n^{-(k+1)}
	\le 2/2^{k+1}$ and $|\sin\theta|\ge\frac{2}{\pi}|\theta|$ yield
	$|\delta_k|\le\pi/2^{k+1}$, confirming $\delta_k\to0$ and
	$a_k\to\tfrac12$.
	
	We now expand~\eqref{eq:delta_eq} systematically.  Write
	$\Theta:=\pi\delta_k$.  The bound above gives $\Theta=O(2^{-(k+1)})$.
	For  small $\Theta$, each $\cos$ term contributes
	$1+O(\Theta^2)$ and each $\sin$ term contributes
	$(2m-1)\Theta+O(\Theta^3)$.  Hence the $\sin$ terms enter at order
	\[
	\frac{(2m-1)\Theta}{(2m-1)^{k+1}}
	= \frac{\Theta}{(2m-1)^k}
	= O\!\left(\frac{1}{2^{k+1}(2m-1)^k}\right)
	= O\!\left(\frac{1}{\bigl(2\,(2m-1)\bigr)^{k+1}}\right),
	\]
	i.e.\ they contribute to even denominators $2(2m-1)$
	rather than to the odd denominators $(2m-1)^{k+1}$ themselves.
	Consequently, the expansion of $\Theta$ involves only even-indexed
	terms:
	\begin{equation*}\label{eq: Theta}
		\Theta = \sum_{m=1}^{\infty}\frac{A_{2m}}{(2m)^{k+1}}.
	\end{equation*}
	The coefficients $A_{2m}$ are determined by substituting the Taylor
	expansions
	\[
	\sin\Theta = \sum_{i=0}^{\infty} \dfrac{(-1)^i\Theta^{2i+1}}{(2i+1)!},\quad
	\cos(2m\Theta)=\sum_{i=0}^{\infty} \dfrac{(-1)^i(2m\Theta)^{2i}}{(2i)!},\quad
	\sin((2m-1)\Theta)=\sum_{i=0}^{\infty} \dfrac{(-1)^i \big((2m-1)\Theta\big)^{2i+1}}{(2i+1)!},
	\]
	into \eqref{eq:delta_eq}.  Comparing the coefficients of $1/(2m)^{k+1}$ gives
	\begin{equation*}
		A_2 = 1,\,\, A_4 = -1,\,\, A_6 = 4,\,\, A_8 = -\frac{17}{6},\,\, \dots.
	\end{equation*}
	This completes the proof.
\end{proof}

For the smoothest B-spline case $\mu=k-1$, Lemma~\ref{lem:bspline_coeff}
shows that the leading-order error polynomial on the reference interval
$[-1,1]$ is proportional to $F_k$.  By $F_k^{(s)} = F_{k-s}$, the
leading term of $e^{(s)}$ is proportional to $F_{k-s}$, so the
superconvergence points of $e^{(s)}$ are precisely the zeros of
$F_{k-s}$.  These zeros are characterized by Lemma~\ref{lem:F_zeros}.

\begin{theorem}\label{thm:bspline_points}
Let $\mu=k-1$ (smoothest B-splines) and $k\ge2$.  For an interior element
mapped to $[-1,1]$, the superconvergence points of $e^{(s)}(x)$
($0\le s\le k$ when the point lies inside the element, $0\le s\le \mu$
when it is a node) are the zeros of $F_{k-s}$ (see Lemma~\ref{lem:F_zeros} for their structure):
\begin{itemize}
  \item If $k-s$ is even, $F_{k-s}$ has simple zeros at $-1,0,1$,
    so $e^{(s)}(x)$ exhibits superconvergence at the element endpoints
    and the midpoint.
  \item If $k-s$ is odd, $F_{k-s}$ has exactly two simple symmetric
    zeros $\pm a_{k-s}\in(-1,1)$, so $e^{(s)}(x)$ exhibits superconvergence
    at two symmetric interior points.  These points approach $\pm1/2$
    as $k-s\to\infty$ with the asymptotics given in
    Lemma~\ref{lem:F_zeros}.
\end{itemize}
There are no other superconvergence points in $[-1,1]$.
\end{theorem}
\begin{remark}
	Theorem~\ref{thm:bspline_points} also holds for $\mu=k-2$ ($k\ge2$).
	Indeed, when $\mu=k-2$, Lemma~\ref{lem:galerkin_count} still gives
		$N_{\text{Gal}}=0$ (cf.\ the calculation leading to~\eqref{eq,N,G}),
	so all leading-order relations originate from the nodal
	superconvergence constraints.  Theorem~\ref{thm:local_super} applies
	at the nodes for all $s\equiv k\pmod{2}$ with $0\le s\le k-2$,
	which are precisely the same derivative orders as in the
	$\mu=k-1$ case (the additional order $s=k-1$ never enters because
	$k-1\not\equiv k\pmod{2}$).  Consequently the triangular system
	in the proof of Lemma~\ref{lem:bspline_coeff} is unchanged, and the
	leading error polynomial remains proportional to $F_k$.
\end{remark}
	\begin{remark}
		Denote by $e_k(x)$ the numerical solution error using the
		B-spline space of degree $k$.
		Theorem~\ref{thm:bspline_points} shows that $e_{k_1}^{(s_1)}(x)$
		has the same superconvergence points as $e_{k_2}^{(s_2)}(x)$ if
		$k_1-s_1=k_2-s_2$.  If we organize the superconvergence points by
		increasing $k$ in one direction and increasing $s$ in the other,
		then all entries on the same diagonal (same $k-s$) have
		identical superconvergence points, which is consistent with the
		computationally obtained superconvergence point
		table~\cite[Table~1]{Anitescu2015}.  This interesting pattern had
		not been pointed out before, and it is now theoretically confirmed
		by our analysis.
	\end{remark}
% ===========================================================================
% 5. Numerical experiments
% ===========================================================================
\section{Numerical experiments}
\label{sec:numerical}

In this section, we compute the leading-order coefficient ratios
for selected $(k,\mu)$ pairs and provide numerical verification of
our main theoretical results. 
 In Subsection~\ref{sec:leading-coeff},
based on Theorem~\ref{thm:closure},
we compute the ratios $c_j\,/\,c_{k+1}$ appearing in the asymptotic
expansion of the finite element error for selected pairs $(k,\mu)$. In addition,
for pairs  $(k,\mu)$ with $\mu=k-1$, we compute the natural superconvergence points and verify their asymptotic behavior predicted by
Lemma~\ref{lem:F_zeros}.
In Subsection~\ref{sec:num-verify}, the natural superconvergence points
of the numerical solution are confirmed for the spline spaces
$S^{k,\mu}_h$ corresponding to some $(k,\mu)$ pairs.

% ----------------------------------------------------------------------
\subsection{Computation  of leading-order coefficient ratios and natural superconvergence points}
\label{sec:leading-coeff}

Following Section~\ref{sec:asymptotic}, for those elements $I_i$ where $|x_0-x_{i+1/2}|=O(h)$ and $x_0$ 
is the symmetric center of an $O(h|\ln h|)$ mesh, the asymptotic expansion of the numerical solution error is:
\begin{equation*}
	e(x)|_{I_i} = \sum_{j\in\mathcal{U}_k} c_j L_j(t)+O(h^{k+2}),
\end{equation*}
where $\mathcal{U}_k$ is given by \eqref{eq:U_set}.
For even $k$ we have $0\notin\mathcal{U}_k$.
For odd $k$ the index $0$ belongs to $\mathcal{U}_k$, but
Lemma~\ref{lem:c0_super} shows that $c_0$ is actually of higher
order (it vanishes for $\mu\le1$); consequently $c_0$ can be absorbed into
the  remainder of higher order. We focus on the effective leading-order polynomial
\[
\mathcal{L}_{k,\mu}(t)=\sum_{j\in\mathcal{U}_k\setminus\{0\}} c_j L_j(t),
\]
the ratios $c_j\,/\,c_{k+1}$ of which can be determined by Theorem~\ref{thm:closure}. We calculate the leading-order coefficient ratios and the theoretical superconvergence points as follows.

First, we select pairs  $(k,\mu)=: (5,2)$, $(6,3)$, $(7,4)$, $(8,3)$, $(8,5)$. For these pairs, the Galerkin orthogonality constraints together with the nodal
superconvergence constraints form the closure system of Theorem~\ref{thm:closure}.
\begin{itemize}
	\item $(k,\mu)=(5,2)$.
	Here $\mathcal{U}_k\setminus\{0\}=\{2,4,6\}$, $N_{\text{Gal}}=1$,
	$N_{\text{Nod}}=1$.
	By \eqref{eq:alpha_def} and \eqref{eq:alpha_lambda}, the Galerkin
	orthogonality constraints give
	\[
	c_2+c_4+c_6 = \alpha_1 = 2\lambda_0 L_1(1),\qquad
	c_4+c_6 = \alpha_3 = 2\lambda_0 L_3(1).
	\]
	Therefore $c_2=0$.
	By \eqref{eq: superconvergence constraints}, the nodal
	superconvergence constraints give
	\[
	c_2L'_2(1)+c_4L'_4(1)+c_6L'_6(1)=0.
	\]
	Hence the ratios are
	\[
	c_2:c_4:c_6 = 0:-L'_6(1):L'_4(1)=0:-21:10,
	\]
	and the leading-order  asymptotic expansion is $\mathcal{L}_{5,2}= (c_{6}/10)(-21L_4+10L_6)$.
	
	\item $(k,\mu)=(6,3)$.
	Here $\mathcal{U}_k\setminus\{0\}=\{1,3,5,7\}$, $N_{\text{Gal}}=1$,
	$N_{\text{Nod}}=2$.
	Equations~ \eqref{eq:alpha_def} and \eqref{eq:alpha_lambda}  give
	\begin{equation*}
		c_3=\alpha_2-\alpha_4=(L'_2(1)-L'_4(1))2\lambda_1=-14\lambda_1,\qquad c_5+c_7=\alpha_4 = 20\lambda_1.
	\end{equation*}
	The nodal  superconvergence constraints \eqref{eq: superconvergence constraints} with $s=0,2$ give
	\begin{align*}
		c_1+c_3+c_5+c_7 = 0,\qquad
	c_3L''_3(1)+c_5L''_5(1)+c_7L''_7(1)	=15c_3+105c_5+378c_7 = 0.
	\end{align*}
	Solving yields
	\[
	c_1:c_3:c_5:c_7 = 39:91:(-175):45,
	\]
	and the leading-order asymptotic expansion $\mathcal{L}_{6,3}=  (c_{7}/45) (39L_1+91L_3-175L_5+45L_7)$.
	\item $(k,\mu)=(7,4)$.
	Here $\mathcal{U}_k\setminus\{0\}=\{2,4,6,8\}$, $N_{\text{Gal}}=1$,
	$N_{\text{Nod}}=2$.
	Using equations~ \eqref{eq:alpha_def} and \eqref{eq:alpha_lambda} for the Galerkin orthogonality constraints and \eqref{eq: superconvergence constraints} for the nodal  superconvergence constraints yields
	\[
	c_2:c_4:c_6:c_8 = 53:318:(-219):35,
	\]
   	and the corresponding leading-order asymptotic expansion	follows accordingly.
	\item $(k,\mu)=(8,3)$.\
	Here $\mathcal{U}_k\setminus\{0\}=\{1,3,5,7,9\}$, $N_{\text{Gal}}=2$,
	$N_{\text{Nod}}=2$.
	Using equations~ \eqref{eq:alpha_def} and \eqref{eq:alpha_lambda} for the Galerkin orthogonality constraints and \eqref{eq: superconvergence constraints} for the nodal  superconvergence constraints yields
	\[
	c_1:c_3:c_5:c_7:c_9 = 102:238:374:(-1085):371,
	\]
	and the corresponding leading-order asymptotic expansion	follows accordingly.
	\item $(k,\mu)=(8,5)$.\
	Here $\mathcal{U}_k\setminus\{0\}=\{1,3,5,7,9\}$, $N_{\text{Gal}}=1$,
	$N_{\text{Nod}}=3$.
	Using equations~ \eqref{eq:alpha_def} and \eqref{eq:alpha_lambda} for the Galerkin orthogonality constraints and \eqref{eq: superconvergence constraints} for the nodal  superconvergence constraints yields
	\[
		c_1:c_3:c_5:c_7:c_9 = 13260:21794:(-51986):19005:(-2073),
	\]
	and the corresponding leading-order asymptotic expansion
	follows accordingly.
\end{itemize}

\begin{table}[htbp]
	\centering
	\caption{The coefficient ratios of $F_k=\Lop^{k}(L_1)$, and the verification of the zeros  and asymptotic expressions $\tilde{a}_k=1/2+1/(\pi 2^{k+1}) - 1/(\pi 4^{k+1})
		+ 4/(\pi 6^{k+1}) - 17/(6\pi 8^{k+1})$ (for odd $k$, corresponding to the positive zero $a_k$).}
	\label{tab:Fk_coeff}
		\begin{tabular}{c|c|c|c|c}
			$k$ & The coefficient ratios of $F_k=\Lop^{k}(L_1)$ & zeros & $|a_k-\tilde{a}_k|$ & $\dfrac{|a_k-\tilde{a}_k|}{|a_{k+2}-\tilde{a}_{k+2}|}$ \\ \hline
			2 & $c_1:c_3=-1:1$               & $-1,0,1$ & --- & --- \\
			3 & $c_2:c_4=-10:3$   & $\pm0.51932962$ & $8.36\times10^{-5}$ & $72.0$ \\
			4 & $c_1:c_3:c_5=6:(-7):1$      & $-1,0,1$ & --- & --- \\
			5 & $c_2:c_4:c_6=33:(-12):1$    & $\pm0.50491857$ & $1.16\times10^{-6}$ & $88.5$ \\
			6 & $c_1:c_3:c_5:c_7=-1287:1547:(-275):15$
			& $-1,0,1$ & --- & --- \\
			7 & $c_2:c_4:c_6:c_8=-650:246:(-26):1$
			& $\pm0.50123923$ & $1.31\times10^{-8}$ & $96.3$ \\
			8 & $c_1:c_3:c_5:c_7:c_9=6630:(-8024):1496:(-105):3$
			& $-1,0,1$ & --- & --- \\
			9 & $c_2:c_4:c_6:c_8:c_{10}=446386:(-170544):19019:(-952):21$
			& $\pm0.50031057$ & $1.36\times10^{-10}$ & 99.9 \\
		\end{tabular}
\end{table}
Second, we select pairs  $(k,\mu)$ where $2\leq k\leq 9$ and $\mu=k-1$.  By Lemma~\ref{lem:bspline_coeff},  we have
\begin{equation*}
	\mathcal{L}_{k,k-1}(t) = CF_k(t)=C\Lop^{k}(L_1)(t).
\end{equation*}
Then, the leading-order coefficient ratios are clear. We also pay attention to the distribution of their zeros, as predicted by Lemma \ref{lem:F_zeros}.  Table~\ref{tab:Fk_coeff} lists the coefficient ratios of $F_k=\Lop^{k}(L_1)$ for $k=2,\dots,9$ (e.g.\ $c_2:c_4=-10:3$ means $F_3=(c_4/3)(-10L_2+3L_4)$), together with the zeros of $F_k$.  For odd $k$, the positive zero $a_k$ is  approximated by the four-term asymptotic expansion
\[
\tilde a_k=\frac12+\frac{1}{\pi 2^{k+1}}-\frac{1}{\pi 4^{k+1}}+\frac{4}{\pi 6^{k+1}}-\frac{17}{6\pi 8^{k+1}}.
\]
The difference $|a_k-\tilde{a}_k|$ is shown in the fourth column of Table~\ref{tab:Fk_coeff}.  The last column displays the ratio $|a_k-\tilde a_k|/|a_{k+2}-\tilde a_{k+2}|$, which approaches $10^2=100$ as $k$ increases, consistent with the $O(1/10^{\,k+1})$ truncation error.

% ----------------------------------------------------------------------
\subsection{Numerical verification of the natural superconvergence points}
\label{sec:num-verify}

We show how to construct a basis of the discrete space $S^{k,\mu}_h$ on a uniform mesh of $\Omega=[0,1]$.  We use an open knot vector
\[
\underbrace{0,\dots,0}_{k+1} < t_1 < t_2 < \cdots < t_{N-1}
< \underbrace{1,\dots,1}_{k+1},
\]
where the $k+1$-fold repetition at each end forces the splines to
interpolate the boundary values, and the interior knots
$t_i = i/N$ ($i=1,\dots,N-1$) are uniform and distinct.
The standard Cox--de~Boor recurrence \cite{Hughes2005} then generates B-spline basis
functions that are automatically $C^{k-1}$ across every interior knot;
this is the maximally smooth space $S^{k,k-1}_h$.
A lower smoothness $\mu<k-1$ is enforced by raising the multiplicity
of the interior knots: if an interior knot has multiplicity
$m = k-\mu$, the smoothness there drops to $C^{k-m}=C^{\mu}$.
Thus, repeating every interior knot $k-\mu$ times yields the space
$S^{k,\mu}_h$.

\begin{example}
	As a test example, we consider the convection--diffusion--reaction
	equation
	\[
	-u''+u'+u=f\quad\text{on }(0,1),\qquad u(0)=u(1)=0,
	\]
	with the manufactured non-periodic solution
	$u(x)=\sin(\pi x)e^{x}$.
	The discrete solution $u_h\in S^{k,\mu}_h$ is obtained by the standard
	Galerkin method on a uniform mesh of size $h=1/N$.
\end{example}

Let $u_h$ ($h=1/N$) and $u_{h/2}$  be successive
Galerkin solutions.  For a number $m\in[-1,1]$, the pointwise
convergence rate of the $s$-th derivative is defined by
\begin{equation*}\label{eq:order}
r_{s,\Omega}(m) :=
\log\!\left(\frac{
\max_{I'_i\subset\Omega}
\bigl|(u-u_{h})^{(s)}\bigl(\tfrac{1-m}{2}x'_{i-1}+\tfrac{1+m}{2}x'_i\bigr)\bigr|
}{\max_{I''_i\subset\Omega}
\bigl|(u-u_{h/2})^{(s)}\bigl(\tfrac{1-m}{2}x''_{i-1}+\tfrac{1+m}{2}x''_i\bigr)\bigr|}\right)/
\log 2,
\end{equation*}
where $I'_i$ (resp.\ $I''_i$) are the elements of the coarse (resp.\ fine)
mesh.  We evaluate $r_{s,\Omega_{in}}(m)$ on the inner region
$\Omega_{in}=[0.2,0.8]$ and $r_{s,\Omega_{b}}(m)$
on the boundary region $\Omega_{b}=[0,1]\setminus[0.2,0.8]$, where
the local mesh symmetry condition of Theorem~\ref{thm:local_super} fails
(the symmetric region $B_d(x_0)$ cannot be contained in $\Omega$), so the
superconvergence is expected to be destroyed there.
Clearly, the ideal value of $r_{s,\Omega_{in}}(m)$ and $r_{s,\Omega_{b}}(m)$ is $k+1-s$. When the actual value of $r_{s,\Omega_{in}}(m)$ or $r_{s,\Omega_{b}}(m)$ exceeds $k+1-s$,  the point $m$ corresponds to a superconvergence point for each element scaled to [-1,1] in $\Omega_{in}$ or $\Omega_{b}$.

Figure~\ref{fig:inner} displays the convergence rates for the selected pairs
$(k,\mu)=(5,2),(5,4),(8,3),(8,5),(8,7)$.
In every case the observed superconvergence points coincide with those
predicted by the leading-order error polynomials computed in
Subsection~\ref{sec:leading-coeff}:
\begin{itemize}
	\item $(k,\mu)=(5,2)$:\ $r_{0,\Omega_{in}}$ peaks at the zeros of
	$\mathcal{L}_{5,2}=-21L_4+10L_6$,
	and $r_{1,\Omega_{in}}$ at the zeros of $\mathcal{L}_{5,2}'$.
	\item $(k,\mu)=(5,4)$:\ $r_{0,\Omega_{in}}$ peaks at the zeros of $F_5=\Lop^{5}(L_1)$,
	$r_{1,\Omega_{in}}$ at the zeros of $F_5'=F_4$, and $r_{2,\Omega_{in}}$ at the zeros of $F''_5=F_3$.
	\item $(k,\mu)=(8,3)$:\ $r_{0,\Omega_{in}}$ and $r_{1,\Omega_{in}}$ peak at the
	zeros of $\mathcal{L}_{8,3}=102L_1+238L_3+374L_5-1085L_7+371L_9$ and $\mathcal{L}_{8,3}'$, respectively.
	\item $(k,\mu)=(8,5)$:\ $r_{0,\Omega_{in}}$ and $r_{1,\Omega_{in}}$ peak at the
	zeros of $\mathcal{L}_{8,5}=13260L_1+21794L_3-51986L_5+19005L_7-2073L_9$ and $\mathcal{L}_{8,5}'$, respectively.
	\item $(k,\mu)=(8,7)$:\ $r_{0,\Omega_{in}}$ peaks at the zeros of $F_8=\Lop^{8}(L_1)$,
	$r_{1,\Omega_{in}}$ at the zeros of $F_7$, and $r_{3,\Omega_{in}}$
	at the zeros of $F_5$.
\end{itemize}
These results fully confirm the theoretical analysis of
Theorems~\ref{thm:closure} and~\ref{thm:bspline_points}.
In contrast, the boundary-region rates $r_{s,\Omega_{b}}$ stay near
$k+1-s$ without any noticeable local increase for all tested pairs, confirming
that the failure of the local symmetry condition indeed destroys the
superconvergence there.

	\begin{center}
		\includegraphics[width=140pt]{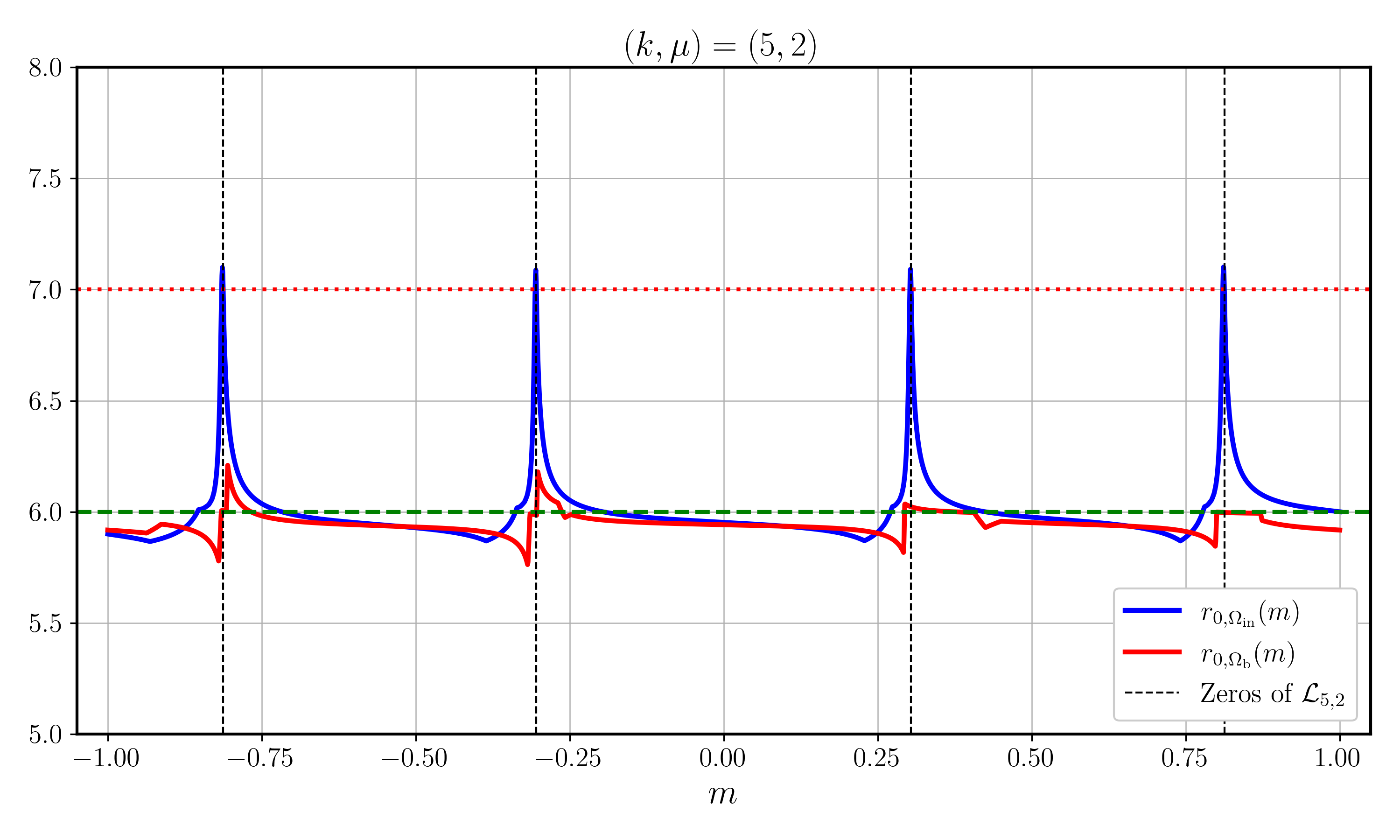}
		\includegraphics[width=140pt]{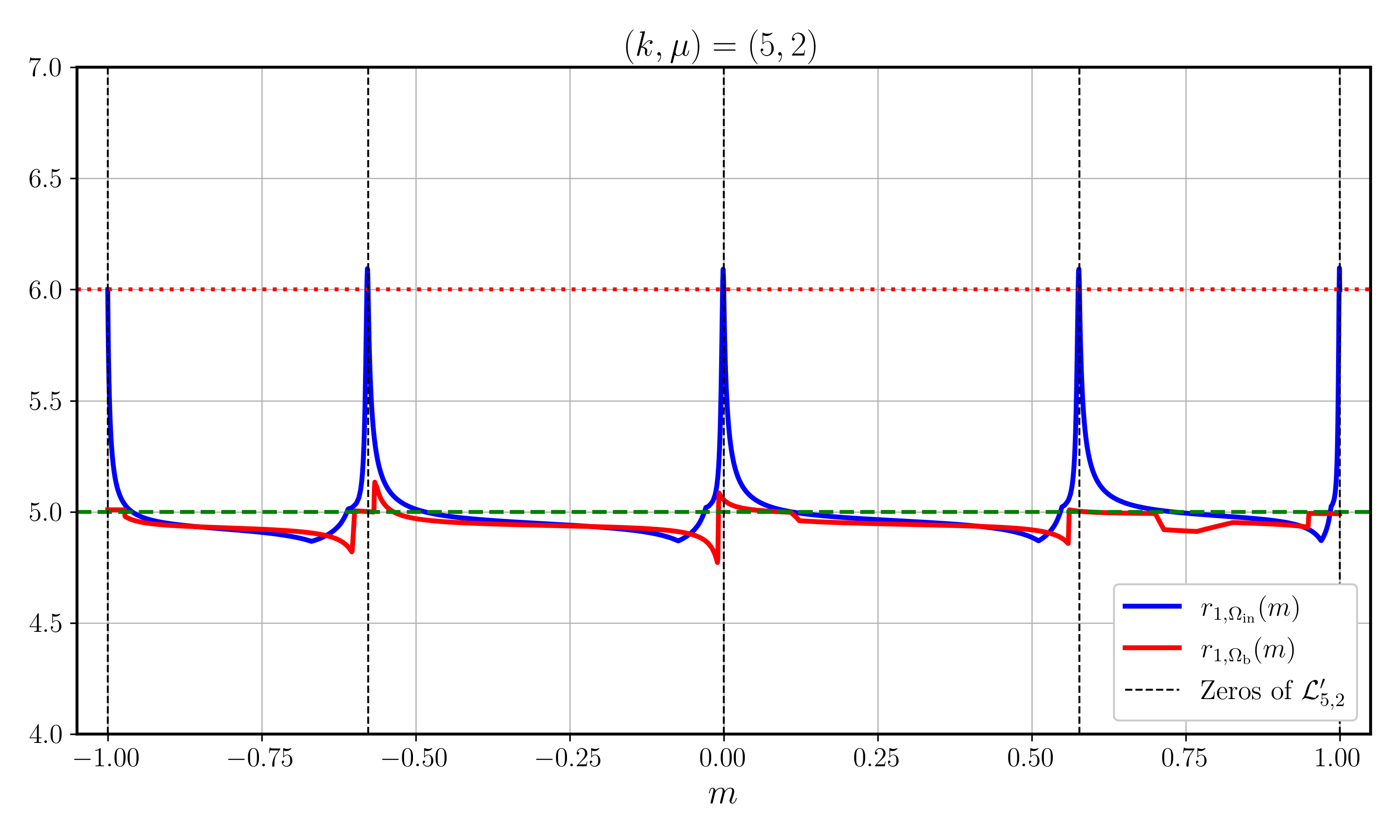}
		\includegraphics[width=140pt]{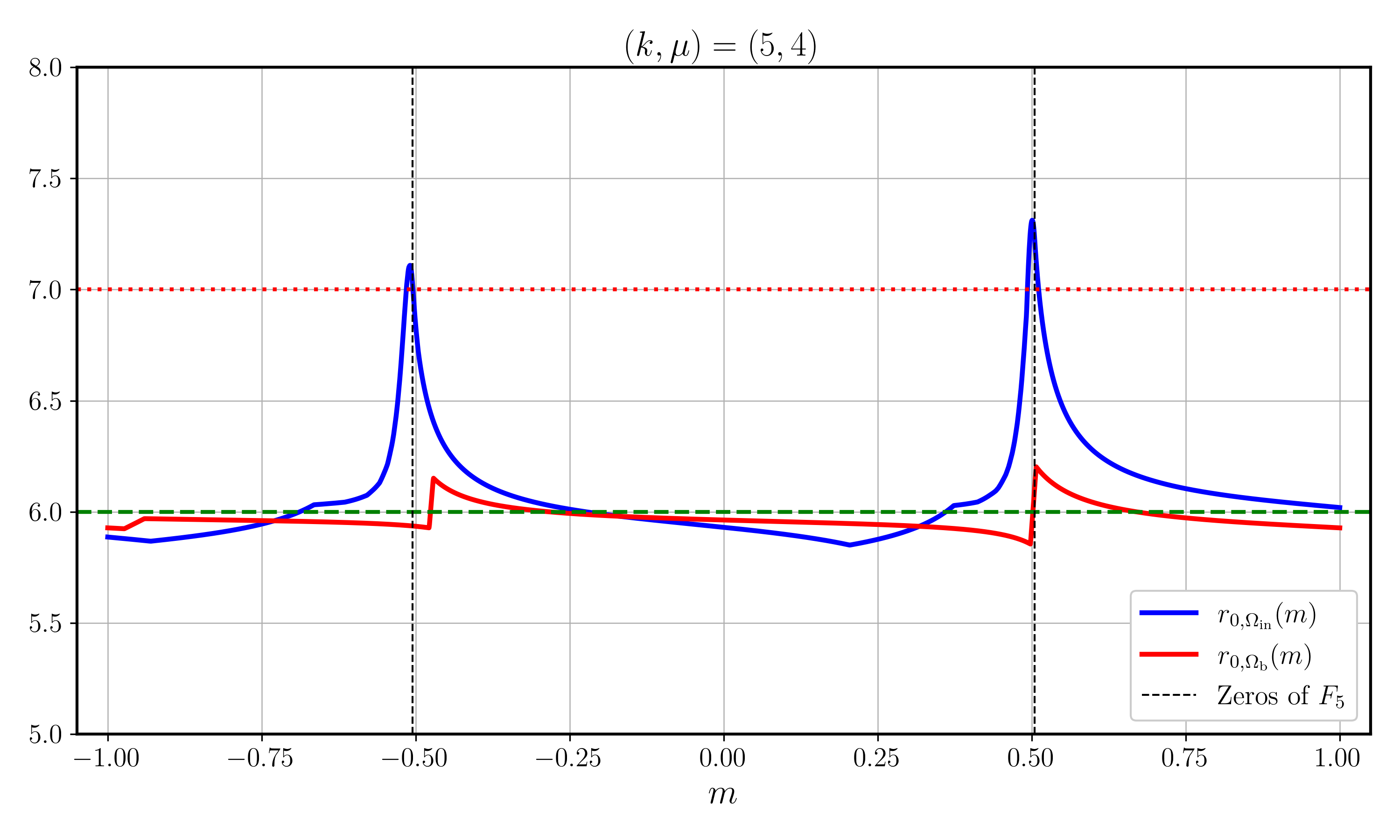}

		\smallskip
		\includegraphics[width=140pt]{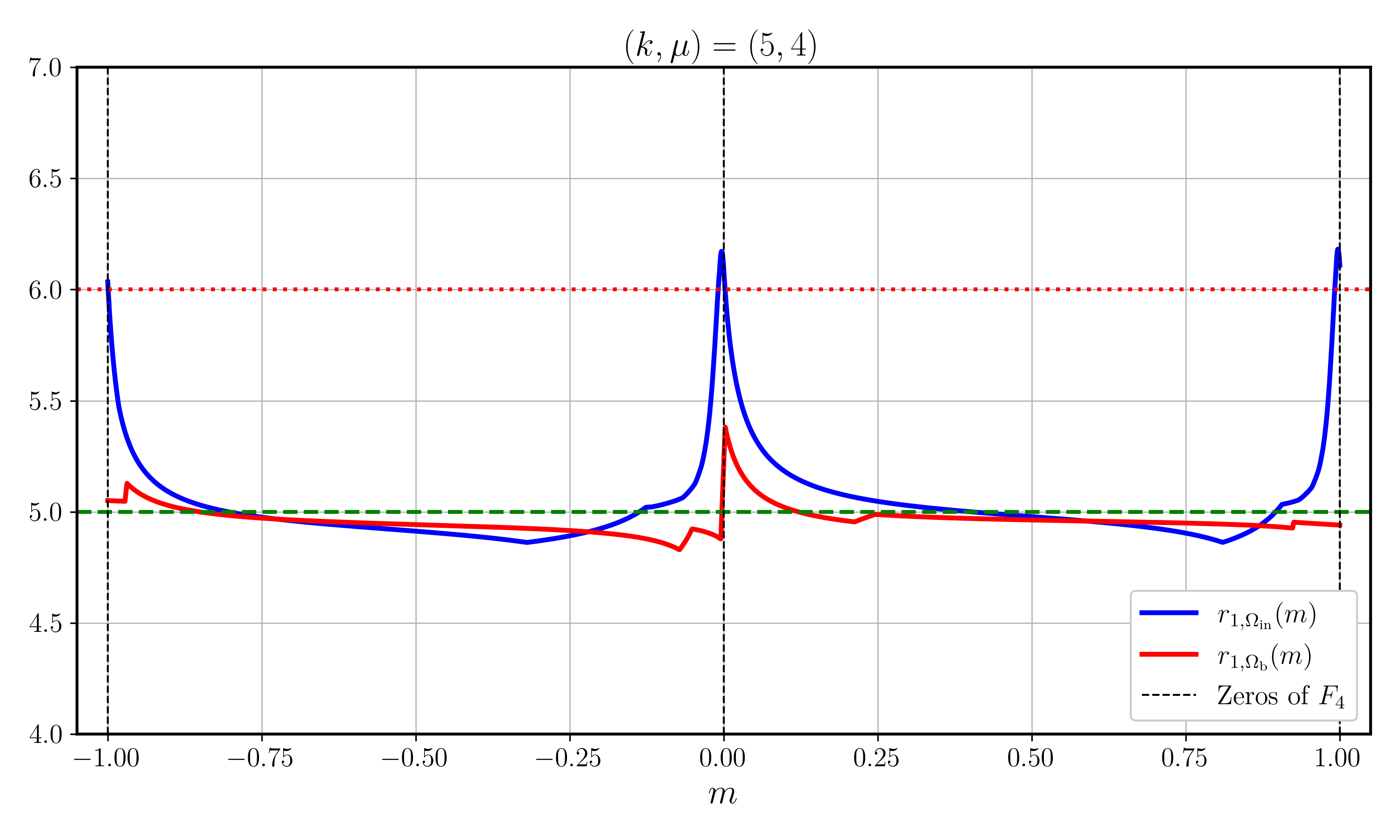}
		\includegraphics[width=140pt]{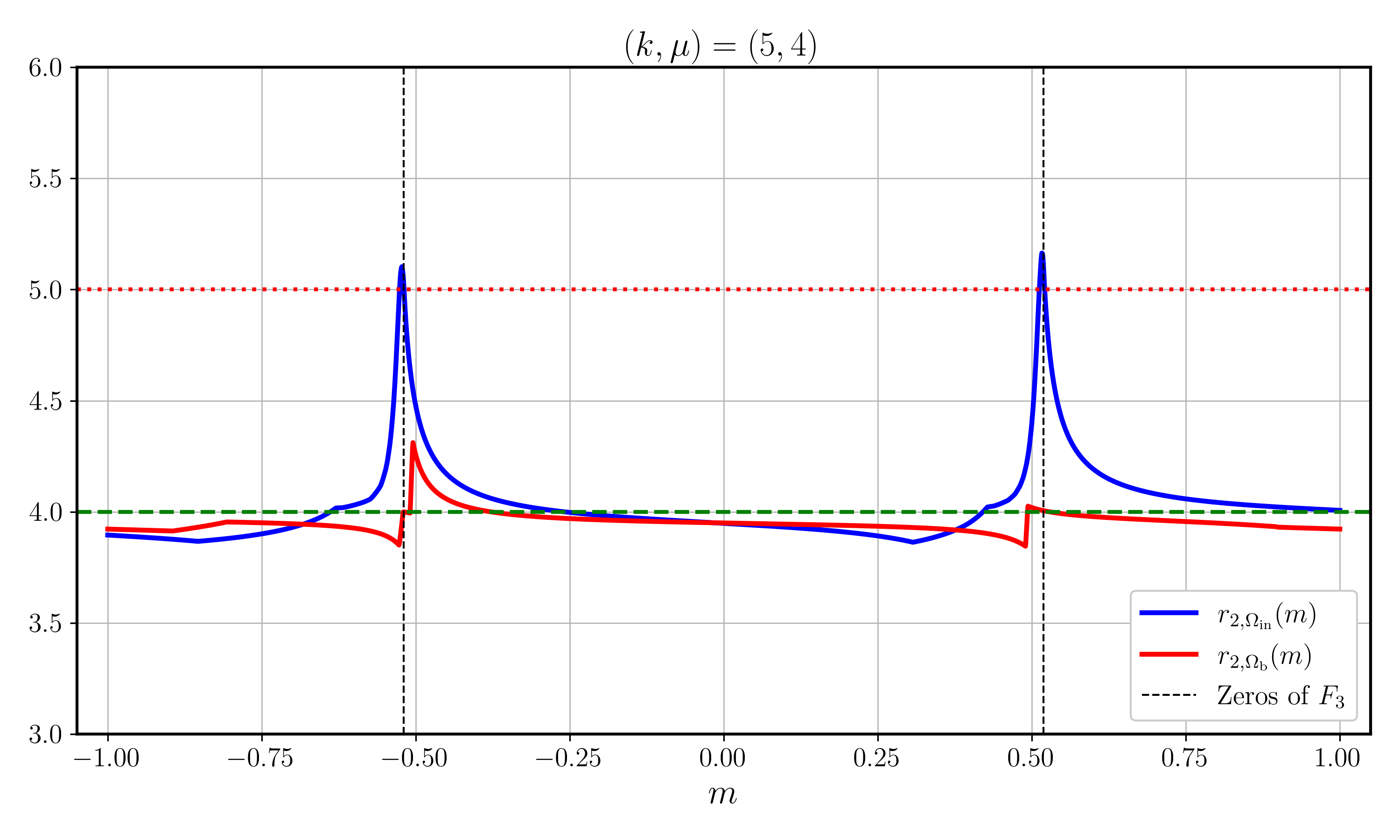}
		\includegraphics[width=140pt]{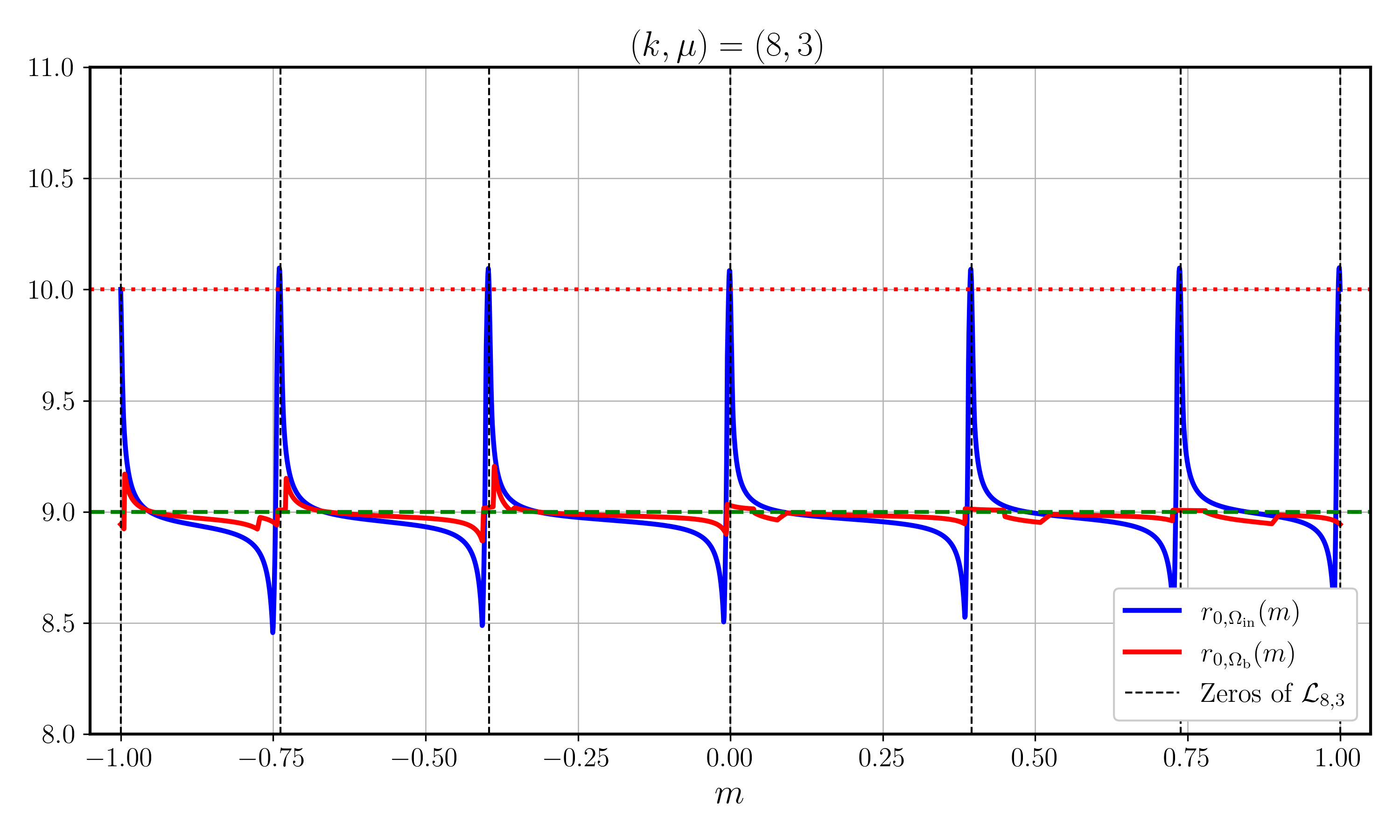}

		\smallskip
		\includegraphics[width=140pt]{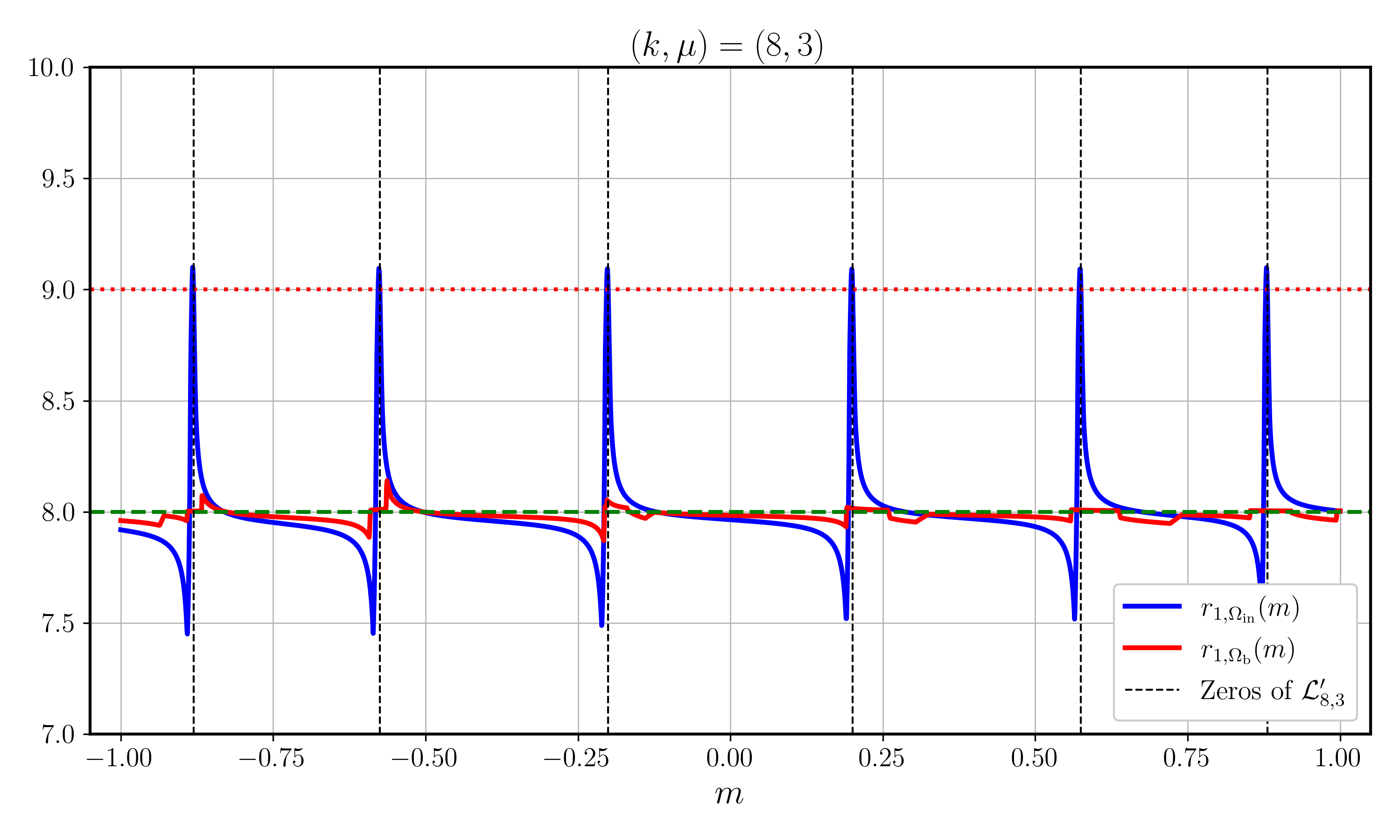}
		\includegraphics[width=140pt]{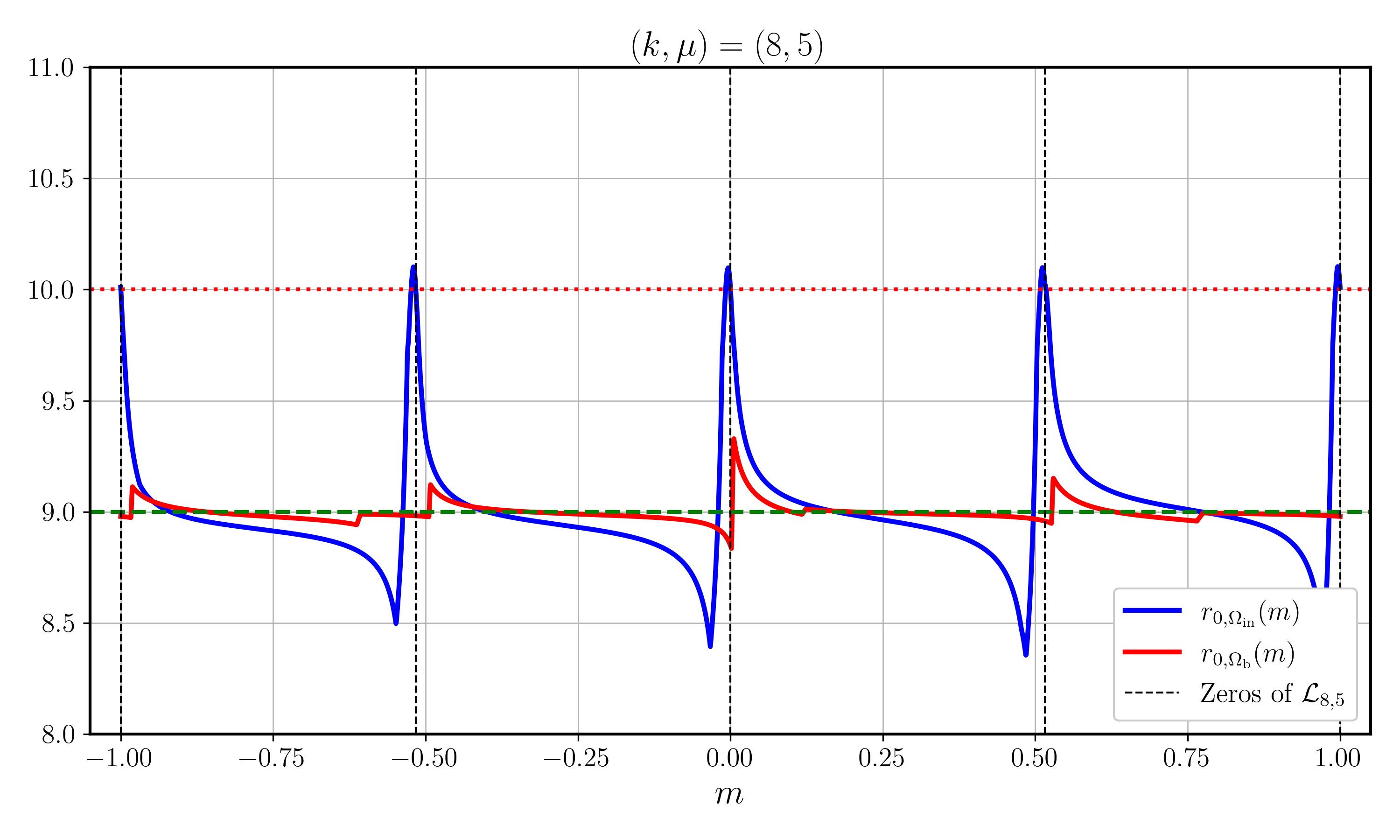}
		\includegraphics[width=140pt]{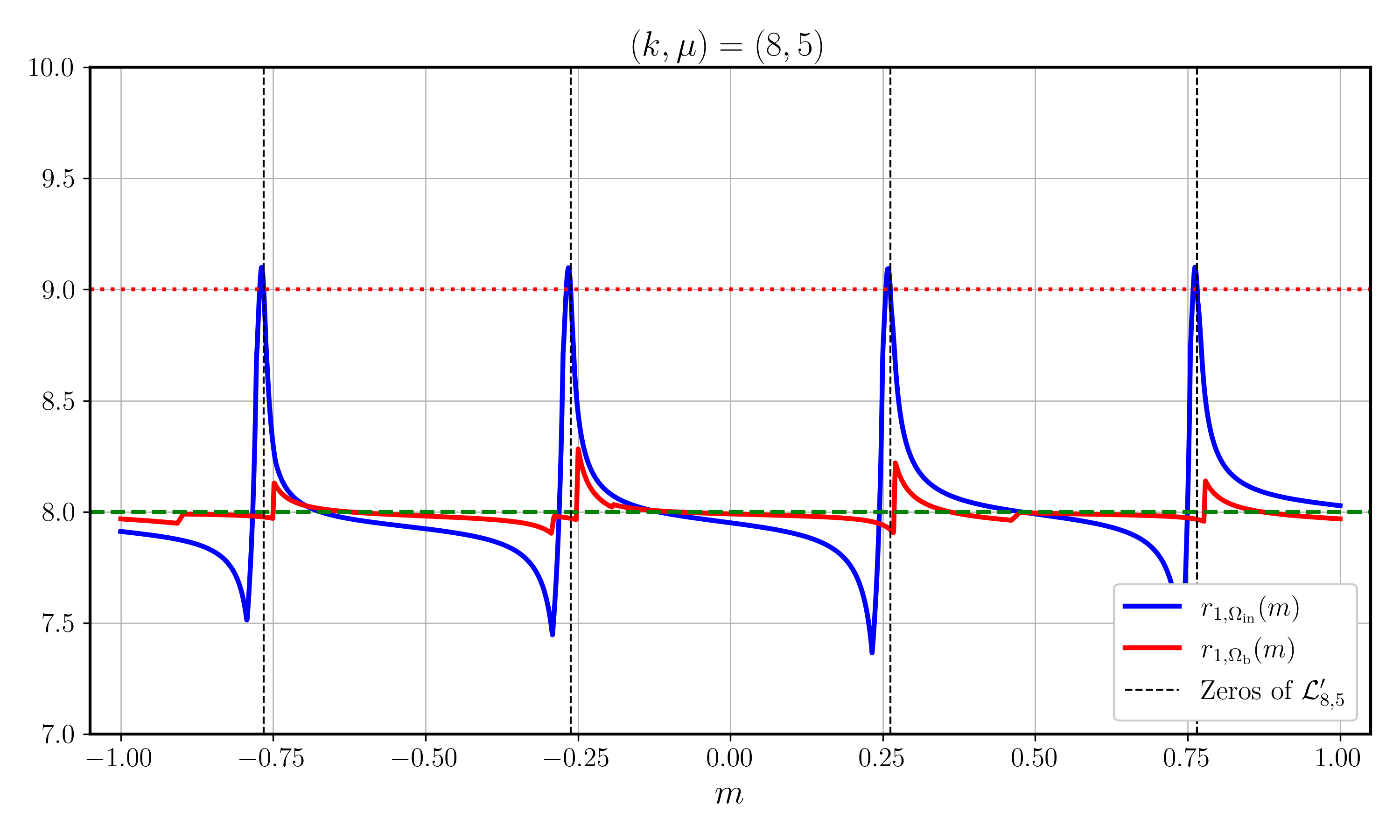}

		\smallskip
		\includegraphics[width=140pt]{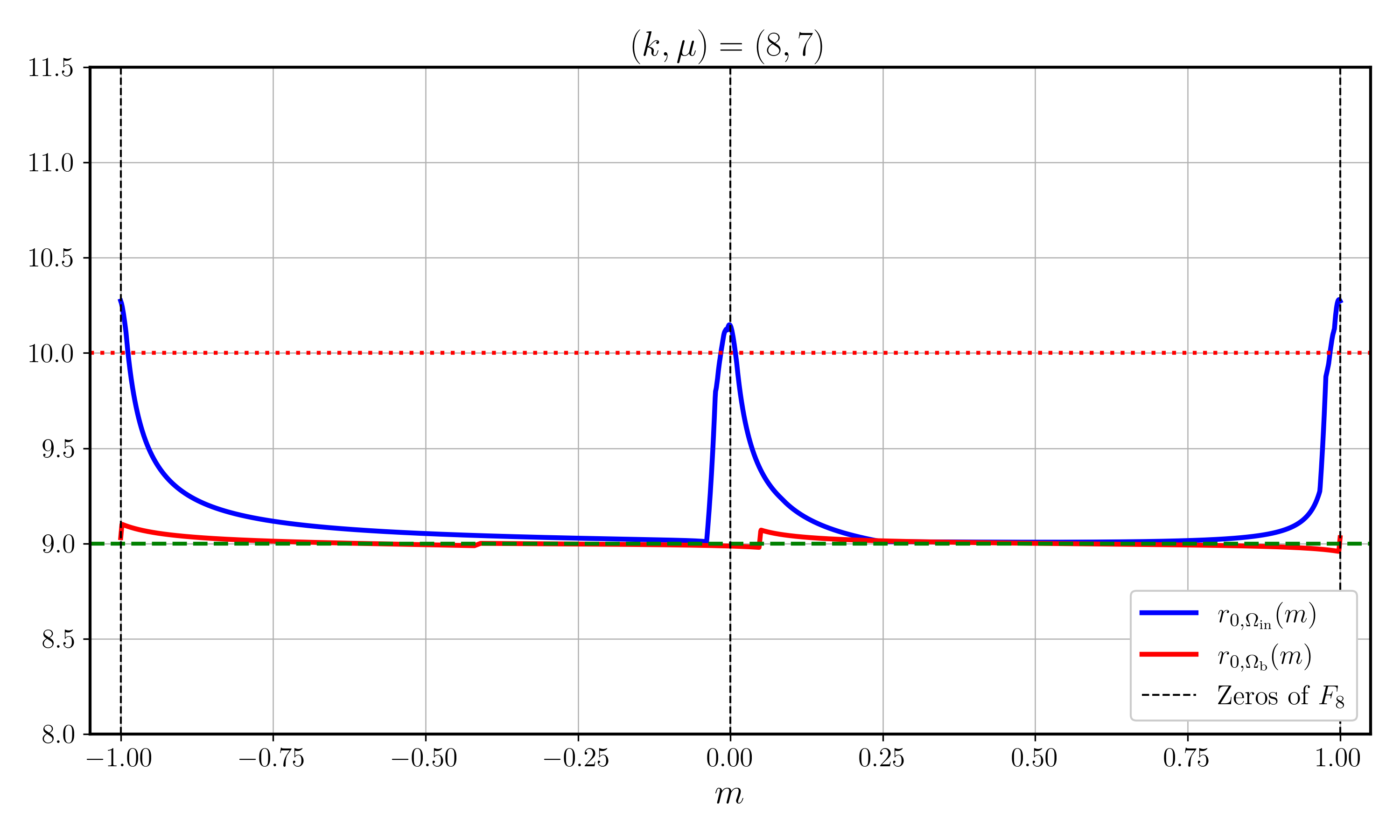}
		\includegraphics[width=140pt]{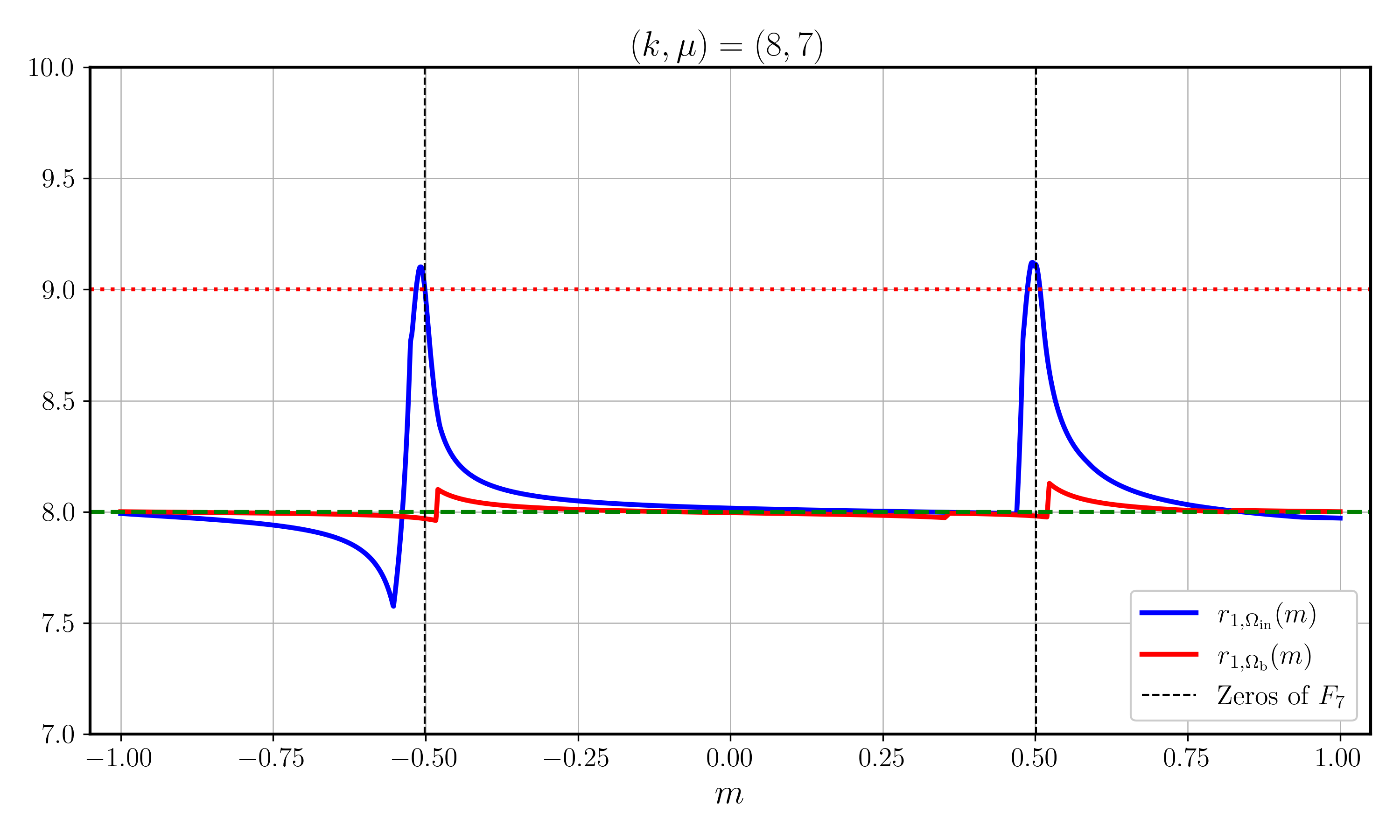}
		\includegraphics[width=140pt]{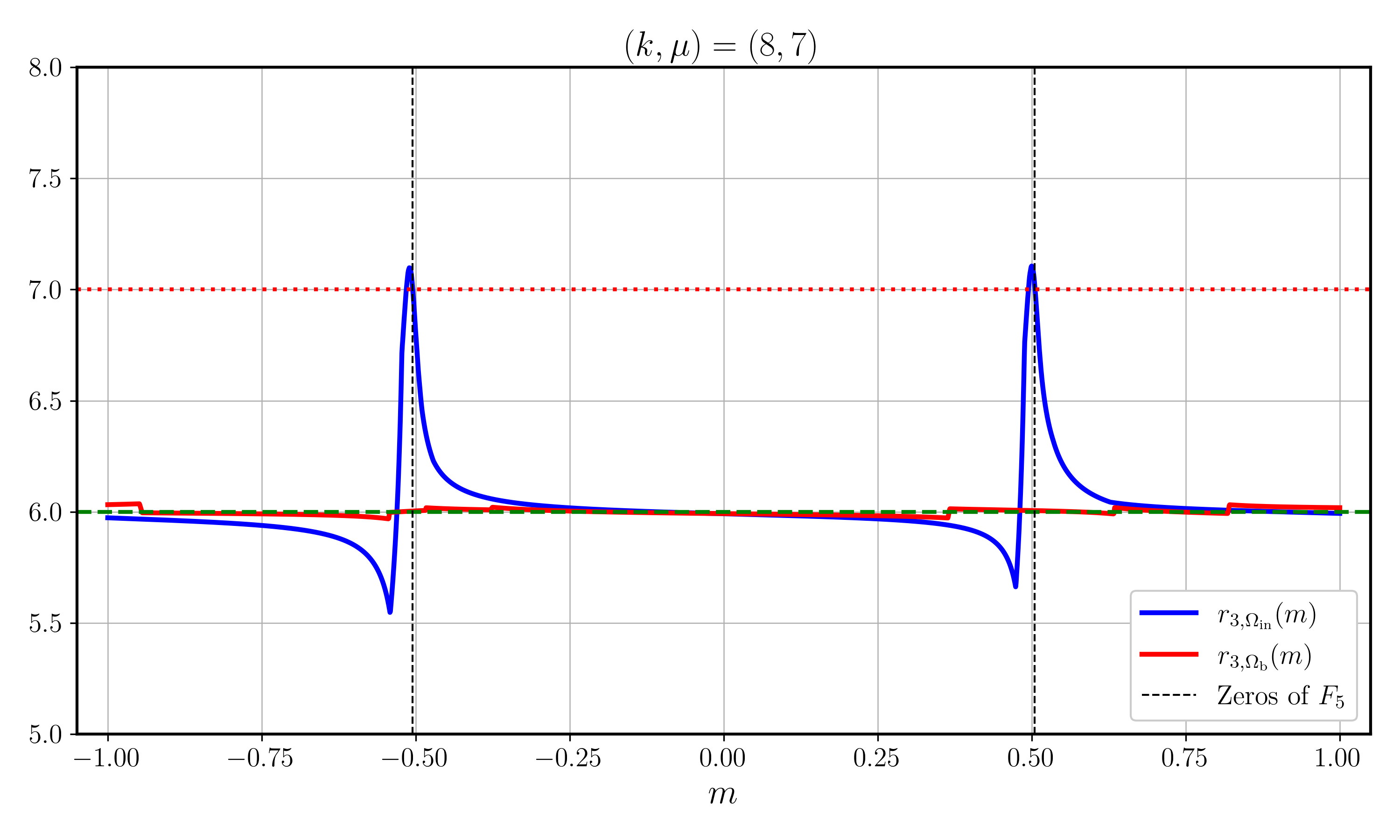}
	\end{center}
	\captionof{figure}{Convergence rates $r_{s,\Omega_{in}}(m)$ and $r_{s,\Omega_{b}}(m)$  for selected pairs $(k,\mu)=(5,2), (5,4), (8,3), (8,5), (8,7)$ and some $s$. The rates are computed from $N=40$ to $80$ for $k=5$ and $(k,\mu)=(8,7)$, and from $N=20$ to $40$ for $(k,\mu)=(8,3), (8,5)$.}
	\label{fig:inner}

% ===========================================================================
% 6. Conclusion
% ===========================================================================
\section{Conclusion}
\label{sec:conclusion}

We have presented a systematic analysis of the natural
superconvergence points for one-dimensional spline finite element approximations of second-order elliptic problems.  A closure theorem was established that
determines the asymptotic expansion of the error by
combining the Galerkin orthogonality with the superconvergence
conditions.  For the maximal smoothness cases $\mu=k-1$
(B-splines) and $\mu=k-2$, the Galerkin orthogonality  conditions vanish and the
asymptotic expansion is obtained through a simple antiderivative
recurrence on Legendre polynomials.
When $k-s$ is even, the $s$-th derivative of the error is superconvergent at the
element endpoints and midpoint; when $k-s$ is odd, it is
superconvergent at two symmetric interior points, and no other superconvergence points exist.  The interior positive superconvergent point  approaches $1/2$ as odd $k-s\to\infty$, and
its asymptotic expansion was derived explicitly.  Numerical
experiments for several $(k,\mu)$ pairs confirm the theoretical predictions.

Several directions for future work are of interest.  An immediate
extension is to tensor-product grids in two and three dimensions.
Although numerical experiments suggest that the superconvergence
points on tensor-product meshes are essentially guided by the
one-dimensional structure, a rigorous proof is non-trivial.  Another natural direction is the analysis of
superconvergence points for higher-order equations, such as
fourth-order problems, where the structure of the leading error
polynomials may differ from the second-order case studied here.
In the longer term, the analysis of natural superconvergence points
on simplicial meshes remains a challenging open problem that would
significantly advance the understanding of superconvergence in
unstructured grids.

% ===========================================================================
% Appendix A: Sharp estimate of the projected Taylor remainder
% ===========================================================================
\appendix
\section{Sharp estimate of the projected Taylor remainder for $s\ge2$}
\label{app:lnh}

In this appendix we prove the sharp estimate of the projected Taylor
remainder $(\Pi_{B_d}r)^{(s)}(x_0)$ for $s\ge2$ stated in
Lemma~\ref{lem:app_sharp} below, with a constant independent of $h$ and
$\ln h$.  Combined with the Taylor decomposition and the parity argument
in the proof of Lemma~\ref{lemma,2}, this estimate yields the sharp
local approximation bound \eqref{lemma, u-pihu} for $k-s$ even.  The
proof is based on local discrete delta functions.

Consider $s\geq 2$. Define the local discrete delta 
$\delta_{B_d}=\delta_{B_d}^{s,x_0} \in\intr S_h^{k-1,\mu-1}(B_d)$ by
\begin{equation}
(\chi,\delta_{B_d})_{L^2(B_d)} = \chi^{(s-1)}(x_0),
\qquad
\forall\,\chi\in\intr S_h^{k-1,\mu-1}(B_d).
\label{eq:app_local_delta}
\end{equation}
Here, $s\le\mu$ when $x_0$ is a mesh node, and $s\le k$ when $x_0$ lies in the interior of an element.

\begin{lemma}\label{lem:decay, local}
Let $s\geq 2$. There exist constants $C,c>0$, independent of $h$ and
  $x_0$, such that
  \begin{equation}
      |\delta_{B_d}(x)|\le C\,h^{-s}\,e^{-c|x-x_0|/h},\qquad
      \forall x\in\Omega.
      \label{eq:delta_decal, local}
  \end{equation}
  For any  prescribed $M>0$,
      \begin{equation}
          m:=\int_{B_d}\delta_{B_d}\,dx = O(h^M),
      \label{eq: estimate m}
  \end{equation}
  provided $C^*$ in $d=C^*h|\ln h|$  is   sufficiently large.
\end{lemma}

\begin{proof}
  The space $\intr S_h^{k-1,\mu-1}(B_d)$ consists of splines that vanish at the boundary nodes $N_b$.
  The exponential decay property of Lemma~\ref{lem:decay} remains valid
  for $\delta_{B_d}$, because the mass matrix of
  $\intr S_h^{k-1,\mu-1}(B_d)$ is a principal submatrix of the global
  mass matrix and therefore satisfies the Demko--Moss--Smith theorem with
  the same constants.  Consequently,
  \begin{equation*}
      |\delta_{B_d}(x)|\le C h^{-s}e^{-c|x-x_0|/h}.
  \end{equation*}
  Because $1\notin\intr S_h^{k-1,\mu-1}(B_d)$, we cannot directly conclude that
  $\int_{B_d}\delta_{B_d}=0$ from~\eqref{eq:app_local_delta}. To overcome this, we construct a smooth cutoff function
  $\psi_1\in\intr S_h^{k-1,\mu-1}(B_d)$ with $\psi_1=1$ for
  $|x-x_0|\le d/2$, $\psi_1=0$ for $|x-x_0|\ge 3d/4$.  Such a function exists because $B_d$ is a union of
  whole elements of size $C^*h|\ln h|$ and the spline space contains a
  partition of unity on the interior.  Then $\psi_1^{(s-1)}(x_0)=0$ for
  $s\ge2$, so by definition~\eqref{eq:app_local_delta},
  \[
  \int_{B_d}\psi_1\,\delta_{B_d}\,dx = \psi_1^{(s-1)}(x_0)=0.
  \]
  Writing $1=\psi_1+(1-\psi_1)$ and using the exponential decay of
  $\delta_{B_d}$, we obtain
  \begin{equation*}
      m:=\int_{B_d}\delta_{B_d}\,dx
      = \int_{B_d}\psi_1\,\delta_{B_d}\,dx
      + \int_{|x-x_0|\ge d/2}(1-\psi_1)\,\delta_{B_d}\,dx = O(h^M),
  \end{equation*}
  for any prescribed $M>0$, provided $C^*$  is
  sufficiently large.  
\end{proof}

Let $r$ denote the Taylor remainder of $u$ about $x_0$:
\[
r(x)=u(x)-\sum_{j=0}^{k+1}\frac{u^{(j)}(x_0)}{j!}(x-x_0)^j,
\qquad r^{(j)}(x_0)=0\;(j\le k+1).
\]
As shown in the proof of Lemma~\ref{lemma,2}, for $k-s$ even the local
approximation error reduces to
$(u-\Pi_{B_d}u)^{(s)}(x_0)=-(\Pi_{B_d}r)^{(s)}(x_0)$.  Hence the local approximation bound
\eqref{lemma, u-pihu} follows once the estimate
\eqref{eq:app_sharp_bound} below is available.

\begin{lemma}\label{lem:app_sharp}
Let $C^*$ of $d=C^*h|\ln h|$ be sufficiently large and $s\geq 2$. Then
\begin{equation}
|(\Pi_{B_d}r)^{(s)}(x_0)| \le  Ch^{k+2-s},
\label{eq:app_sharp_bound}
\end{equation}
with a constant independent of $h$ and $\ln h$.  Here, $s\le\mu$ when
$x_0$ is a mesh node, and $s\le k$ when $x_0$ lies in the interior of an
element.
\end{lemma}

\begin{proof}
Set $\eta:=(\Pi_{B_d}r)'\in S_h^{k-1,\mu-1}(B_d)$.  At the boundary
nodes $y\in N_b$, $\Pi_{B_d}r$ matches $r$ in value and derivatives up
to order $\mu$, so $\eta$ does not vanish on $\partial B_d$ and hence
$\eta\notin\intr S_h^{k-1,\mu-1}(B_d)$.  To circumvent this, split
$\eta$ according to the support of the B-spline basis:
\begin{equation*}
\eta = \sum_{i\in I}d_i\tilde\psi_i
+\sum_{i\in J}d_i\tilde\psi_i:= \eta_0+\eta_b,
\end{equation*}
where indices $i\in I$ correspond to basis functions whose supports
lie entirely in the interior of $B_d$, and $i\in J$ the remaining
functions whose supports touch $\partial B_d$.  Thus, 
\begin{equation*}
  \eta_0\in \intr S_h^{k-1,\mu-1}(B_d).
\end{equation*}
Since the
supports of boundary basis functions do not contain $x_0$, 
\begin{equation*}
\eta_b^{(s-1)}(x_0)=0.
\end{equation*}
Consequently, by \eqref{eq:app_local_delta},
\begin{equation}\label{eq: Pi_{B_d}r}
(\Pi_{B_d}r)^{(s)}(x_0)=\eta^{(s-1)}(x_0)=\eta_0^{(s-1)}(x_0)
=(\eta_0,\delta_{B_d}).
\end{equation}
Using exponential decay \eqref{eq:delta_decal, local} of $\delta_{B_d}$, we have
\begin{equation}\label{eq: eta_0, inner}
(\eta_0,\delta_{B_d}) = (\eta,\delta_{B_d})
- (\eta_b,\delta_{B_d})
= (\eta,\delta_{B_d}) + O(h^M).
\end{equation}
The Ritz projection orthogonality in~\eqref{eq:local_projection} gives
\begin{equation}\label{eq: local_projection,eta}
  (\eta,\psi)=(r',\psi), \quad \forall \psi\in W_h:=\{\chi':\chi \in \intr{S}_{h}^{k,\mu}(B_d)\}.
\end{equation}
To exploit this orthogonality, we would like to choose
$\psi=\delta_{B_d}$, but $\delta_{B_d}\notin W_h$ because
$\int_{B_d}\delta_{B_d}=m\neq0$ (cf.~\eqref{eq: estimate m}), while
every $\psi\in W_h$ has zero integral (it is a derivative of a function
vanishing on $\partial B_d$).  We therefore construct a corrected
discrete delta $\widetilde\delta_{B_d}\in W_h$ that differs from
$\delta_{B_d}$ by a negligible $O(h^M)$ perturbation.
Let $\phi_0$ be a B-spline basis function of
$\intr S_h^{k-1,\mu-1}(B_d)$ with $x_0\in\operatorname{supp}(\phi_0)$, and define
\begin{equation*}
\psi_0(x):=\Bigl(\int\phi_0\Bigr)^{-1}\phi_0(x),
\qquad \int\psi_0=1.
\end{equation*}
Using the estimate~\eqref{eq: estimate m} for $m=\int_{B_d}\delta_{B_d}$, set
\begin{equation*}
\widetilde\delta_{B_d}:= \delta_{B_d} - m\psi_0.
\end{equation*}
By construction, $\widetilde\delta_{B_d}\in \intr S_h^{k-1,\mu-1}(B_d)$ and
$\int_{B_d}\widetilde\delta_{B_d}=0$.  Consequently,
$\widetilde\delta_{B_d}\in W_h$ and differs from $\delta_{B_d}$ only by
the $O(h^M)$ term $m\psi_0$.

Now writing
$\delta_{B_d}=\widetilde\delta_{B_d}+m\psi_0$ and using
$\widetilde\delta_{B_d}\in W_h$ in \eqref{eq: local_projection,eta}, it follows from \eqref{eq: Pi_{B_d}r} and \eqref{eq: eta_0, inner} that
\begin{align*} 
(\Pi_{B_d}r)^{(s)}(x_0)=(\eta,\delta_{B_d})+ O(h^M)
&= (\eta,\widetilde\delta_{B_d}) + m(\eta,\psi_0)+ O(h^M) \\
&= (r',\widetilde\delta_{B_d}) + m(\eta,\psi_0) + O(h^M)\\
&= (r',\delta_{B_d}) - m(r',\psi_0) + m(\eta,\psi_0)+ O(h^M)\\
&= (r',\delta_{B_d}) +  O(h^M).
\end{align*}
The terms $m(r',\psi_0)$ and $m(\eta,\psi_0)$ are absorbed into the
$O(h^M)$ remainder because $m=O(h^M)$ by the estimate~\eqref{eq: estimate m}. To estimate  $(r',\delta_{B_d}) $, we write the remainder as
\[
r'(x) = \frac{u^{(k+2)}(t)}{(k+1)!}\,(x-x_0)^{k+1},
\]
where $t$ lies between $x$ and $x_0$. Then
\[
|(r',\delta_{B_d})|
\leq  C|u|_{W^{k+2,\infty}(B_d)}\,
\int_{B_d}|x-x_0|^{k+1}\,|\delta_{B_d}|\,dx.
\]
Fix $C_0>0$ and split
\begin{align*}
  \int_{B_d}|x-x_0|^{k+1}\,|\delta_{B_d}|\,dx &\leq  \int_{|x-x_0|\le C_0 h} |x-x_0|^{k+1}\,|\delta_{B_d}|\,dx
  + \int_{|x-x_0|> C_0 h} |x-x_0|^{k+1}\,|\delta_{B_d}| \,dx\\
  &=: J + K.
\end{align*}
For the tail, using exponential decay \eqref{eq:delta_decal, local} of $\delta_{B_d}$,
\[
K\le C h^{-s}\int_{C_0 h}^\infty r^{k+1} e^{-cr/h}\,dr
     = C h^{k+2-s} \int_{C_0}^\infty y^{k+1} e^{-cy}\,dy
     \le C h^{k+2-s}\int_{0}^\infty s^{k+1} e^{-s}\,ds
     = C\Gamma(k+2) h^{k+2-s},
\]
where we use the gamma function 
\begin{equation*}
  \Gamma(z) =  \int_{0}^\infty y^{z-1} e^{-y}\,dy.
\end{equation*}
For the main part, noticing that
$|x-x_0|\le C_0 h$ gives $|(x-x_0)^{k+1}|\le (C_0 h)^{k+1}$, and
$\int_{|x-x_0|\le C_0 h}|\delta_{B_d}(x)|\,dx\le C h^{-s+1}$, we have
\[
J\le C h^{k+2-s}.
\]
Hence $|(r',\delta_{B_d})| \le C h^{k+2-s}$, and consequently
\[
|(\Pi_{B_d}r)^{(s)}(x_0)|\le Ch^{k+2-s},
\]
with no logarithmic factor, provided $C^*$ of $d=C^*h|\ln h|$ is sufficiently large.
\end{proof}

% ===========================================================================
% Acknowledgements
% ===========================================================================
\section*{Acknowledgements}
This work was partially supported by the National Natural Science Foundation
of China (Grant No.~12501537).

% ===========================================================================
% Bibliography
% ===========================================================================
%% Loading bibliography style file
\bibliographystyle{cas-model2-names}

% Loading bibliography database
\bibliography{cas-refs}

\end{document}